\documentclass[11pt]{article}
\usepackage{amssymb}
\usepackage{amsfonts}
\usepackage{amsmath}
\usepackage{cases}
\usepackage[usenames]{color}
\usepackage{mathrsfs}
\usepackage{cite}
\usepackage{cases}
\usepackage{amsthm}
\usepackage{indentfirst}
\usepackage{setspace}
\usepackage{abstract}
\usepackage[none]{hyphenat}

\allowdisplaybreaks[4]
\def\ptl{\partial}

\def\cl{\centerline}

\def\al{\alpha}

\def\b{\beta}

\def\vs{\vspace*}

\def\Z{\mathbb{Z}}
\def\N{\mathbb{N}}

\def\C{\mathbb{C}}
\def\g{\mathbf{g}}

\def\C{\mathbb{C}}

\def\SUM#1#2{{\mbox{$\sum\limits_{#1}^{#2}$}}}

\numberwithin{equation}{section}
\newtheorem{theo}{Theorem}[section]
\newtheorem{defi}[theo]{Definition}

\newtheorem{lemm}[theo]{Lemma}

\newtheorem{prop}[theo]{Proposition}
\newtheorem{clai}{Claim}

\newtheorem{remark}[theo]{Remark}

\begin{document}
\sloppy{}
\baselineskip 6pt
\lineskip 6pt

\begin{center}{\Large\bf Simple modules over the Takiff Lie algebra for $\mathfrak{sl}_{2}$}
\footnote {This paper was written during the visit of the author to Uppsala University as a guest PhD student.
The hospitality of Uppsala University and the financial support of  China Scholarship Council (No. 202006260124) are greatly appreciated.
\\ \indent \ \ Corresponding author: Xiaoyu Zhu (1810079@tongji.edu.cn)}
\end{center}
\vs{6pt}

\cl{Xiaoyu Zhu}
\cl{\footnotesize School of Mathematical Sciences, Tongji University, Shanghai 200092, China}
\cl{\footnotesize E-mail: 1810079@tongji.edu.cn}
\quad\\
\vs{12pt} \par

\noindent{{\bf Abstract.}} In this paper, we construct, investigate and, in some cases,
classify several new classes of (simple) modules over the Takiff $\mathfrak{sl}_{2}$.
More precisely, we first explicitly construct and classify, up to isomorphism, all
modules over the Takiff $\mathfrak{sl}_{2}$ that are
$U\left(\overline{\mathfrak{h}}\right)$-free of rank one. These split into three general families of modules.
The sufficient and necessary conditions for simplicity of these modules are presented,
and their isomorphism classes are determined. Using the vector space duality and
Mathieu's twisting functors, these three classes of modules are used to
construct new families of weight modules over the Takiff $\mathfrak{sl}_{2}$.
We give necessary and sufficient conditions for these weight modules to be simple
and, in some cases, completely determine their submodule structure.
\vs{6pt}


\section{Introduction}
Classification and construction of modules are two fundamental problems in
contemporary representation theory. In representation theory of Lie algebras,
the problem of classification of all modules is too difficult, in general.
It is thus natural to restrict the classification problem from all
modules to ``smaller'' natural classes of modules.

One of the most natural general classes of modules is that of simple modules.
Unfortunately, in the case of Lie algebras, it turns out that classification
of simple modules is also too difficult, in general. In fact, in the case
of semi-simple complex Lie algebras, the Lie algebra $\mathfrak{sl}_{2}$
is the only algebra for which there is some version of such classification,
see \cite{RB,VM}. At the same time, there are various natural classes of
simple modules over semi-simple Lie algebras that are fairly well-understood.
These include all simple finite dimensional modules (see e.g. \cite{EC,JD}),
simple highest weight modules (cf. \cite{JD,JEH}), simple Whittaker modules
(cf. \cite{BM,BK}), simple Gelfand-Zeitlin modules (cf. \cite{DFO,EMV,Webster})
and simple weight modules with finite dimensional weight spaces (also known as
Harish-Chandra modules), see \cite{SF,M}.

A new family of simple Lie algebra modules, usually called
``$U(\mathfrak{h})$-free modules'', was recently introduced by
J. Nilsson in \cite{N} and, independently, by H. Tan and K. Zhao in \cite{TZ1}.
The modules in this family are free of finite rank when restricted to a
fixed Cartan subalgebra of the Lie algebra in question. In \cite{N},
Nilsson classified simple $U(\mathfrak{h})$-free modules of rank one over the simple
Lie algebra $\mathfrak{sl}_{n+1}$. Moreover, it was shown that such modules
can exist only when the underlying Lie algebra is of type $A$ or $C$ in \cite{N1}.
For the algebra $\mathfrak{sl}_{n+1}$, an alternative construction of such
modules is given in \cite{TZ1}. $U(\mathfrak{h})$-free modules were constructed
and classified in a few other situations, for example, for the Witt algebra
in  \cite{TZ}, for the Virasoro algebra in \cite{CC}, for the
Heisenberg Virasoro algebra in \cite{CG}, see also \cite{CG1,CY,SYZ,YYX,YYX1}
for various other setups.

Semi-simple Lie superalgebras and non-semi-simple Lie algebras are
two natural generalizations of semi-simple Lie algebras. For the first family,
in many cases, the problem of classification of simple supermodules can be
reduced to the problem of classification of simple modules over the corresponding even Lie algebra,
see \cite{CCM} and \cite{CM}. This shows, in particular, that the problem of classification
of simple modules over Lie algebras is important for the study of Lie superalgebras.
For the second family of generalizations, namely, for non-semi-simple Lie algebras,
the problem of classification of simple modules is not yet as well-studied as
in the case of semi-simple Lie algebras. Natural families of non-semi-simple
Lie algebras include, in particular, current Lie algebras,
conformal Galilei algebras, Takiff Lie algebras and others.

Current Lie algebras form a natural generalization of semi-simple Lie algebras
and are studied very intensively due to their connection to Kac-Moody Lie
algebras and quantum groups. For this type Lie algebras, there is a full
classification of simple Harish-Chandra modules, see \cite{ML}. The highest
weight theory over truncated current Lie algebra was investigated in \cite{BJW}.

Takiff Lie algebras are the "smallest" non-semi-simple truncated current Lie algebras.
The Takiff Lie algebra for $\mathfrak{sl}_{2}$ belongs also to  the class of conformal
Galilei algebras (see, e.g. \cite{LMZ}) and is defined as the semidirect product of
$\mathfrak{sl}^{}_{2}$ with its adjoint representation. Such Lie algebras were
first introduced by Takiff in \cite{T}, where the invariant theory for such
Lie algebras was developed. Takiff Lie algebras were further investigated
recently by several people, see \cite{MC,MS} and references therein. In order
to attempt the problem of classification of modules over non-semi-simple Lie algebras,
it is natural to consider first this problem for the Takiff $\mathfrak{sl}_{2}$.

In the present paper we construct, investigate and, in some cases, classify
several new classes of (simple) modules over the Takiff $\mathfrak{sl}_{2}$.
First, we explicitly construct and classify,  up to isomorphism, all
$U\left(\overline{\mathfrak{h}}\right)$-free modules of rank one. These split into three general
families of modules, which we denote $\Gamma(\lambda,a,b)$,
$\Theta(\lambda,a,b)$ and $\Omega\left(\lambda,b,\beta_{1}\left(\overline{h}\right)\right)$.
Our main results about these modules are collected into the following
two theorems.

\begin{theo}\label{prop-3.2bc}
Let $M$ be a $\g$-module such that the restriction of $U(\g)$ to
$U\left(\overline{\mathfrak{h}}\right)$ is free of rank one. Then $M$ is
isomorphic to either $\Gamma(\lambda,a,b)$, or $\Theta(\lambda,a,b)$ or
$\Omega\left(\lambda,b,\beta_{1}\left(\overline{h}\right)\right)$,
for some $\lambda\in \C^{\times}$, $a,b\in \C$ and
$\beta_{1}\left(\overline{h}\right)\in \C\left[\,\overline{h}\,\right]$, which
are defined in Definition \ref{defi01}.
\end{theo}

\begin{theo}\label{prop-3.2a} Let $\lambda\in\C^{\times}$, $a,b\in\C$ and
$\beta_{1}\left(\overline{h}\right)\in \C\left[\,\overline{h}\,\right]$. Then the following holds:
\begin{itemize}
  \item [\rm(i)] The $\g$-modules $\Gamma\left(\lambda,a,b\right)$ and $\Theta(\lambda,a,b)$
are simple. We  have $\Gamma(\lambda,a,b)\cong \Gamma(\lambda',a',b')$ if and only if
$\lambda=\lambda'$, $a=a'$ and $b=b'$. Similarly,
$\Theta(\lambda,a,b)\cong \Theta(\lambda',a',b')$ if and only if
$\lambda=\lambda'$, $a=a'$ and $b=b'$.
  \item [\rm(ii)] The $\g$-module
$\Omega\left(\lambda,b,\beta_{1}\left(\overline{h}\right)\right)$ is simple
if and only if $b\neq0$. We have
$\Omega\left(\lambda,b,\beta_{1}\left(\overline{h}\right)\right)\cong \Omega\left(\lambda',b',\beta'_{1}\left(\overline{h}\right)\right)$
if and only if $\lambda=\lambda'$, $b=b'$ and
$\beta_{1}\left(\overline{h}\right)=\beta'_{1}\left(\overline{h}\right)$.
\end{itemize}
\end{theo}

Next, we use the  vector space duality and Mathieu's twisting functors,
to construct, based on $\Gamma(\lambda,a,b)$, $\Theta(\lambda,a,b)$ and
$\Omega\left(\lambda,b,\beta_{1}\left(\overline{h}\right)\right)$,
and investigate three new families, $M_{\al,\b}^{\lambda,a,b}$, $N_{\al,\b}^{\lambda,a,b}$ and $V_{\al,\b}^{\lambda,a,\beta_{1}\left(\overline{h}\right)}$,
of simple weight modules over the Takiff $\mathfrak{sl}_{2}$. The corresponding results are collected in
the following two theorems, combined with Theorems \ref{vztheo4.2} and \ref{theo4.7}.

\begin{theo}\label{vztheo4.1}Let $M_{\al,\b}^{\lambda,a,b}$ be a $\g$-module defined in Proposition $\ref{defi3.1}$. Then the following holds:
\begin{itemize}
  \item[\rm(i)] The module  $M_{\al,\b}^{\lambda,a,b}$ is a simple $\g$-module if and only if $\b^{2}+a\neq0$ or $\b^{2}+a=0$ and $(\al_{k}+2s)\b+b\neq0$, for any $k\in \Z$ and $s\in \N$.
  \item[\rm(ii)] If  $M_{\al,\b}^{\lambda,a,b}$ is not simple, then $M_{k,s}:=U(\g)\eta_{\al_{k},\b}$ is the (opposite) Verma module with the lowest weight $\eta_{\al_{k},\b}$, where $k\in \Z$ is such that $(\al_{k}+2s)\b+b=0$, for some $s\in \N$. Moreover, $M_{\al,\b}^{\lambda,a,b}/M_{k,s}\cong N_{k,s}$, where $N_{k,s}$ is a generalized weight $U(\g)$-module with infinite dimensional weight spaces.
  \item[\rm(iii)] For any $\al,\b,a,b\in \C$ and $\lambda,\lambda'\in \C^{\times}$, we have $M_{\al,\b}^{\lambda,a,b}\cong M_{\al,\b}^{\lambda',a,b}$.
  \item[\rm(iv)] For any $\al,\b,\al',\b',a,b,a',b'\in \C$ and $\lambda,\lambda'\in \C^{\times}$, we have
  $M_{\al,\b}^{\lambda,a,b}\cong M_{\al',\b'}^{\lambda',a',b'}$ if and only if
  $\al-\al'\in 2\Z$, $\b=\b'$, $a=a'$ and $b=b'$.
\end{itemize}\end{theo}

\begin{theo}\label{th0}For any $\al,\b, z\in \C$, we  have $B_{z}M_{\al,\b}^{\lambda,a,b}\cong M_{\al-2z,\b}^{\lambda,a,b}$,
where  $B_{z}$ denotes Mathieu's twisting functor.\end{theo}

Finally, we investigate the connection between these three classes of weight modules
we constructed and obtain the following, a bit surprising, result.

\begin{theo}\label{mzh4}Let $M_{\al,\b}^{\lambda,a,b}$, $N_{\al,\b}^{\lambda,a,b}$ and $V_{\al,\b}^{\lambda,a,\beta_{1}\left(\overline{h}\right)}$ be the $\g$-modules defined in Propositions \ref{defi3.1}, \ref{defi3.21} and \ref{prop4.9}, respectively, where $\lambda\in \C^{\times}$, $\al,\b,a,b\in \C$ and $\b_{1}\left(\overline{h}\right) \in \C\left[\,\overline{h}\,\right]$. Then we have
\begin{itemize}
  \item[\rm(i)] If $\b^{2}+a\neq0$, then $N_{\al,\b}^{\lambda,a,b}\cong M_{\al,\b}^{\lambda,a,b}$.
  \item[\rm(ii)] If $\b+a\neq 0$, then $V_{\al,\b}^{\lambda,a,\beta_{1}\left(\overline{h}\right)}\cong M_{\al,\b}^{\lambda,-a^{2},b}$, where  $b=-2 a\left(\lambda\b_{1}(a)+1\right)$.
\end{itemize} \end{theo}

This paper is organized as follows. In Section 2, we recall some necessary definitions
and preliminary results about $U(\mathfrak{h})$-free modules over the Lie algebra
$\mathfrak{sl}_{2}$. In Section 3, we start with an explicit construction of
$U\left(\overline{\mathfrak{h}}\right)$-free modules of rank one over Takiff $\mathfrak{sl}_{2}$ and
then prove Theorems \ref{prop-3.2bc} and \ref{prop-3.2a}.
In Section 4, we  start by proving Theorem \ref{vztheo4.1}. Then we study the
weight modules associate to the non-weight modules $\Theta(\lambda,a,b)$ and
$\Omega\left(\lambda,b,\beta_{1}\left(\overline{h}\right)\right)$, respectively,
see Theorems \ref{vztheo4.2} and \ref{theo4.7}.
Finally, we prove Theorem \ref{mzh4}.

Throughout this paper, we denote by $\C$, $\C^{\times}$, $\Z$, $\Z_{+}$  and $\N$ the
sets of complex numbers, nonzero complex numbers, integers, non-negative integers and positive integers, respectively.

\section{Preliminaries}
We begin by briefly introducing our conventions. In this paper, all vector spaces (resp. algebras, modules) are defined over $\C$. We denote by $V^{*}$ the dual space of a vector space $V$. Denote by $\C[s,t]$ the polynomial algebra in variables $s$ and $t$ with coefficients in $\C$.

Consider the Lie algebra $\mathfrak{sl}_{2}$ with the standard basis $\{e, h, f\}$ and the Lie bracket
\begin{equation*}
  [e,f]=h,\ \ \ [h,e]=2e,\ \ \ [h,f]=-2f.
\end{equation*}
Let $D:=\C[t]/(t^{2})$ be the algebra of dual numbers. Consider the associated {\it Takiff Lie algebra} $\mathbf{g}:=\mathfrak{sl}_{2}\otimes_{\C}D$ with the Lie bracket
\begin{equation*}
  [a\otimes t^{i}, b\otimes t^{j}]=[a, b]\otimes t^{i+j},
\end{equation*}
where $a, b\in \mathfrak{sl}_{2}$ and $i, j\in \left\{0,1\right\}$, and the Lie bracket on the right hand side being the usual $\mathfrak{sl}_{2}$-Lie bracket. Set
\begin{equation*}
  \overline{e}:=e\otimes t, \ \ \overline{f}:=f\otimes t, \ \ \overline{h}:=h\otimes t.
\end{equation*}
Let $\overline{\mathfrak{n}}_{+}$ be the subalgebra of $\g$ generated by $e, \overline{e}$. Let $\overline{\mathfrak{h}}$ be the subalgebra of $\g$ generated by $h, \overline{h}$, and,  finally,  let $\overline{\mathfrak{n}}_{-}$ be the subalgebra of $\g$ generated by $f, \overline{f}$. Then we have the following triangular decomposition of $\g:$
\begin{equation*}
  \g=\overline{\mathfrak{n}}_{+}\oplus \overline{\mathfrak{h}}\oplus \overline{\mathfrak{n}}_{-}.
\end{equation*}

For a Lie algebra $L$, we denote by $U(L)$ the corresponding universal enveloping algebra. Denote by $U(\g)$-Mod the category of all $U(\g)$-modules.
Let $M\in U(\g)$-Mod, then $M^{*}\in U(\g)$-Mod with the action is defined as follows:
\begin{equation*}
  \left(y\cdot \varphi\right)(v)=-\varphi(y\cdot v) \ \ \ \mbox{for}\ \, y\in U(\g),\, \varphi\in M^{*}, \,v\in M.
\end{equation*}
 Let $\gamma\left(h,\overline{h}\right)\in \C\left[h,\overline{h}\,\right]$ and write $\gamma\left(h,\overline{h}\right)=\sum _{i=0}^{m}\sum_{j=0}^{n}c_{i,j}h^{i}\overline{h}^{j}$ with $m,n\in \Z_{+}$ and $c_{i,j}\in \C$. 
 We define the degree of $h$ and $\overline{h}$ of $\gamma\left(h,\overline{h}\right)$ by $m$ and $n$, respectively. Denote by $\mbox{deg}_{y}\gamma\left(h,\overline{h}\right)$ the degree of $y$ of $\gamma\left(h,\overline{h}\right)$ for $y\in \left\{h,\overline{h}\right\}$.
\begin{defi}{\rm A $\g$-module $M$ is called a} {\it generalized weight module} {\rm provided that} $M=\sum_{\mu\in \overline{\mathfrak{h}}^{*}}M^{\mu}$, where
\begin{equation*}
  M^{\mu}=\left\{ v\in M\,|\, \left(h-\mu(h)\right)^{k}\cdot v=0, \left(\overline{h}-\mu\left(\overline{h}\right)\right)^{s}\cdot v=0,\ \ \mbox{for\ some}\ k,s\in \N\right\}.
\end{equation*}
{\rm The subspace $M^{\mu}$ is called the} generalized weight space {\rm of $M$ associated to $\mu$.}
\end{defi}
\begin{defi}{\rm A $\g$-module $M$ is called a} {\it weight module} {\rm provided that} $M=\sum_{\lambda\in \C} M_{\lambda}$, where
\begin{equation*}
  M_{\lambda}=\left\{ v\in M\,|\, h\cdot v=\lambda v\right\}.
\end{equation*}
{\rm The subspace $M_{\lambda}$ is called the} weight space {\rm of $M$ corresponding to the weight $\lambda$.}
\end{defi}
\begin{defi}{\rm A weight $\g$-module $M$ is called a} {\it highest} {\rm (resp. {\it lowest})} {\it weight module} {\rm with {\it highest} (resp. {\it lowest}) {\it weight} $\lambda\in \C$, if there exists a nonzero weight vector $v\in M_{\lambda}$ such that $M$ is generated by $v$ as a $\g$-module and $\overline{\mathfrak{n}}_{+}\cdot v=0$ (resp. $\overline{\mathfrak{n}}_{-}\cdot v=0$)}.
\end{defi}
\begin{theo}\label{theo1.1}\rm{(see [\ref{CC}, \ref{N}]).} {\it Any $U(\mathfrak{sl}_{2})$-module $M$ such that the restriction of $U(\mathfrak{sl}_{2})$ to $U(\C h)$ is free of rank 1 is isomorphic to one of the modules $$\Delta_{1}(\lambda,a),\,\Delta_{2}(\lambda,a),\,\Delta_{3}(\lambda,a),$$ for some $\lambda\in \C^{\times}$ and $a\in \C$, whose module structures are given as follows}
\begin{align*}
 \Delta_{1}(\lambda,a):\ e\cdot g(h) &=-\frac{1}{\lambda}\left(\frac{h}{2}-a\right)g(h-2), \ \ h\cdot g(h)=hg(h),\\
 f\cdot g(h) &=\lambda\left(\frac{h}{2}+a\right)g(h+2),\\
 \Delta_{2}(\lambda,a):\ e\cdot g(h)&=\lambda g(h-2),\ \ \ h\cdot g(h)=hg(h),  \\
 f\cdot g(h) &=-\frac{1}{\lambda}\left(\frac{h}{2}-a\right)\left(\frac{h}{2}+a+1\right)g(h+2),\\
 \Delta_{3}(\lambda,a):\ e\cdot g(h) &= -\frac{1}{\lambda}\left(\frac{h}{2}+a\right)\left(\frac{h}{2}-a-1\right)g(h-2),\\
   h\cdot g(h)&=hg(h), \ \ \ f\cdot g(h) =\lambda g(h+2),
\end{align*}
where $g(h)\in \C[h]$. Moreover, $\Delta_{1}(\lambda,a)$ is simple if and only if $2a\notin \Z_{+}$, $\Delta_{2}(\lambda,a)$ and $\Delta_{3}(\lambda,a)$ are simple for all $\lambda\in \C^{\times}$ and $a\in \C$.
\end{theo}
\begin{lemm}\label{lemma1.1}For any $\gamma\left(h,\overline{h}\right)\in \C\left[\,h,\overline{h}\,\right]$, we have
\begin{align}
   e\cdot \gamma\left(h,\overline{h}\right)&=\gamma\left(h-2,\overline{h}\right)e\cdot 1-2\overline{\ptl}\left(\gamma\left(h-2,\overline{h}\right)\right)\overline{e}\cdot 1, \label{vze1}\\
  f\cdot \gamma\left(h,\overline{h}\right)&=\gamma\left(h+2,\overline{h}\right)f\cdot 1+2\overline{\ptl}\left(\gamma\left(h+2,\overline{h}\right)\right)\overline{f}\cdot 1,\label{vze2}\\
  \overline{e}\cdot \gamma\left(h,\overline{h}\right)&= \gamma\left(h-2,\overline{h}\right)\overline{e}\cdot 1,\label{vze3}\\
  \overline{f}\cdot \gamma\left(h,\overline{h}\right)&= \gamma\left(h+2,\overline{h}\right)\overline{f}\cdot 1,\label{vze4}
\end{align} where $\overline{\ptl}:=\frac{\ptl}{\ptl \overline{h}}$ is the partial derivative
with respect to $\overline{h}$ on $\C\left[\,h,\overline{h}\,\right]$.
\end{lemm}
\begin{proof}Before we give the proof, we first give the following claim.
\begin{clai}\label{vzcm1}For any $i\in \Z_{+}$, we have \begin{align}
              eh^{i} &=(h-2)^{i}e,\ \ \ \ \ \,e\overline{h}^{i}=\overline{h}^{i}e-2i\overline{h}^{i-1}\overline{e},  \label{vz2.1}\\
              fh^{i} &=(h+2)^{i}f,\ \ \ \ \ f\overline{h}^{i}=\overline{h}^{i}f+2i\overline{h}^{i-1}\overline{f},  \label{vz2.2}\\
              \overline{e}h^{i} &=(h-2)^{i}\overline{e},\ \ \ \ \ \,\overline{e}\,\overline{h}^{i}=\overline{h}^{i}\,\overline{e}, \label{vz2.3}\\
              \overline{f}h^{i} &=(h+2)^{i}\overline{f},\ \ \ \ \  \overline{f}\,\overline{h}^{i}=\overline{h}^{i}\,\overline{f}. \label{vz2.4}
            \end{align}
\end{clai}
We prove this claim by induction on $i$. Noting that \begin{align*}
              eh &= \left[\,e,h\,\right]+he=\left(h-2\right)e, \ \ \ \
              e\overline{h} = \left[\,e,\overline{h}\,\right]+\overline{h}e=\overline{h}e-2\overline{e}, \\
              \overline{e}h &=\left[\,\overline{e},h\,\right]+h\overline{e}=\left(h-2\right)\overline{e},  \ \ \ \
              \overline{e}\,\overline{h} =\left[\,\overline{e},\overline{h}\,\right]+\overline{h}\,\overline{e}=\overline{h}\,\overline{e},
            \end{align*}
thus, (\ref{vz2.1}) and (\ref{vz2.3}) hold for $i=1$. Now suppose that (\ref{vz2.1}) and (\ref{vz2.3}) are true for $i$, then
\begin{align*}
  eh^{i+1} &= \left(h-2\right)^{i}eh=\left(h-2\right)^{i+1}e, \\
  e\overline{h}^{i+1} &= \left(\overline{h}^{i}e-2i\overline{h}^{i-1}\overline{e}\right)\overline{h}=\overline{h}^{i}\left(\overline{h}e-2\overline{e}\right)
  -2i\overline{h}^{i}\overline{e}=\overline{h}^{i+1}e-2(i+1)\overline{h}^{i}\overline{e}, \\
  \overline{e}h^{i+1} &= \left(h-2\right)^{i}\overline{e}h=\left(h-2\right)^{i}\left(h-2\right)\overline{e}=\left(h-2\right)^{i+1}\overline{e},\\
  \overline{e}\overline{h}^{i+1} &= \overline{h}^{i}\overline{e}\,\overline{h}=\overline{h}^{i+1}\overline{e}.
\end{align*}
Hence (\ref{vz2.1}) and (\ref{vz2.3}) are true for all $i$. Similarly, we can obtain (\ref{vz2.2}) and (\ref{vz2.4}). This completes the proof of Claim \ref{vzcm1}. Let $\gamma\left(h,\overline{h}\right)$ be any element in $\C\left[\,h,\overline{h}\,\right]$ and write $\gamma\left(h,\overline{h}\right)=\sum_{i,j}d_{i,j}h^{i}\overline{h}^{j}$ for some $d_{i,j}\in \C$. Then, by Claim \ref{vzcm1}, we get
\begin{align*}
  e\cdot \gamma\left(h,\overline{h}\right)  &=\SUM{i,j}{}d_{i,j}e\cdot h^{i}\overline{h}^{j}\overset{(\ref{vz2.1})}{=}\SUM{i,j}{}d_{i,j} \left(h-2\right)^{i}e\cdot\overline{h}^{j}
   = \gamma\left(h-2,\overline{h}\right)e\cdot 1-2\overline{\ptl}\left(\gamma\left(h-2,\overline{h}\right)\right)\overline{e}\cdot 1, \\
   \overline{e}\cdot \gamma\left(h,\overline{h}\right)&= \SUM{i,j}{}d_{i,j}\overline{e}\cdot h^{i}\overline{h}^{j}\overset{(\ref{vz2.3})}{=}\SUM{i,j}{}d_{i,j} \left(h-2\right)^{i}\overline{e}\cdot\overline{h}^{j}=\gamma\left(h-2,\overline{h}\right)\overline{e}\cdot 1.
\end{align*}
Therefore, (\ref{vze1}) and (\ref{vze3}) are true. Similarly, we can get (\ref{vze2}) and (\ref{vze4}). Consequently, we complete the proof. \end{proof}
\section{Free $U\left(\overline{\mathfrak{h}}\right)$-modules of rank 1 over $\g$}
\subsection{Construction of free $U\left(\overline{\mathfrak{h}}\right)$-modules of rank 1}
The following definition gives a precise construction of a $\g$-module structure on the polynomial algebra $\C\left[\,h,\overline{h}\,\right]$.
\begin{defi}\label{defi01}Let $\lambda\in \C^{\times}$, $a,b\in \C$ and $\gamma\left(h,\overline{h}\right)\in \C\left[\,h,\overline{h}\,\right]$. We can define the $\mathbf{g}$-module structure on $\C\left[\,h,\overline{h}\,\right]$ as follows:
\begin{align*}
  \Gamma\left(\lambda,a,b\right):\ \ \ e\cdot \gamma\left(h,\overline{h}\right)&=-2\lambda \overline{\partial}\left(\gamma\left(h-2,\overline{h}\right)\right),\ \ \ \ \ \ \ \ \ \ \ \ \overline{e}\cdot \gamma\left(h,\overline{h}\right)=\lambda\gamma\left(h-2,\overline{h}\right), \\
  \ \ \ \overline{f}\cdot \gamma\left(h,\overline{h}\right)&=-\frac{1}{4\lambda}\left(\overline{h}^{2}+a\right)\gamma\left(h+2,\overline{h}\right), \ \ \ \ x\cdot \gamma\left(h,\overline{h}\right)=x\gamma\left(h,\overline{h}\right),\\
  \ \ \ f\cdot \gamma\left(h,\overline{h}\right) &= -\frac{1}{2\lambda}\left((h+2)\overline{h}+b\right)\gamma\left(h+2,\overline{h}\right)-\frac{1}{2\lambda}
  \left(\overline{h}^{2}+a\right)\overline{\partial}\left(\gamma\left(h+2,\overline{h}\right)\right),
\end{align*}
\begin{align*}
 \Theta(\lambda,a,b):\ \ \ f\cdot \gamma\left(h,\overline{h}\right)&=2\lambda \overline{\partial}\left(\gamma\left(h+2,\overline{h}\right)\right),\ \ \ \ \ \ \ \ \ \ \ \ \ \ \overline{f}\cdot \gamma\left(h,\overline{h}\right)=\lambda\gamma\left(h+2,\overline{h}\right), \\
  \ \ \ \overline{e}\cdot \gamma\left(h,\overline{h}\right)&=-\frac{1}{4\lambda}\left(\overline{h}^{2}+a\right)\gamma\left(h-2,\overline{h}\right), \ \ \ \ \,x\cdot \gamma\left(h,\overline{h}\right)=x\gamma\left(h,\overline{h}\right),\\
  \ \ \ e\cdot \gamma\left(h,\overline{h}\right) &= -\frac{1}{2\lambda}\left((h-2)\overline{h}+b\right)\gamma\left(h-2,\overline{h}\right)+\frac{1}{2\lambda}
  \left(\overline{h}^{2}+a\right)\overline{\partial}\left(\gamma\left(h-2,\overline{h}\right)\right),
\end{align*}
\begin{align*}
  \Omega\left(\lambda,b,\beta_{1}\left(\overline{h}\right)\right):\ \ \ e\cdot \gamma\left(h,\overline{h}\right)&=\left(\frac{\lambda}{2}h+\al_{1}\left(\overline{h}\right)\right)\gamma\left(h-2,\overline{h}\right)-\lambda\left(\overline{h}+b\right)
  \overline{\partial}\left(\gamma\left(h-2,\overline{h}\right)\right),\\
  \ \ \ f\cdot \gamma\left(h,\overline{h}\right) &= -\left(\frac{1}{2\lambda}h-\beta_{1}\left(\overline{h}\right)\right)\gamma\left(h+2,\overline{h}\right)-\frac{1}{\lambda}\left(\overline{h}-b\right)\overline{\partial}\left(\gamma\left(h+2,\overline{h}\right)\right),\\
  \ \ \ \overline{e}\cdot \gamma\left(h,\overline{h}\right)&=\frac{\lambda}{2}\left(\overline{h}+b\right)\gamma\left(h-2,\overline{h}\right), \ \ \ x\cdot \gamma\left(h,\overline{h}\right)=x\gamma\left(h,\overline{h}\right),\\
  \ \ \ \overline{f}\cdot \gamma\left(h,\overline{h}\right)&=-\frac{1}{2\lambda}\left(\overline{h}-b\right)\gamma\left(h+2,\overline{h}\right),
\end{align*}
where $x\in \overline{\mathfrak{h}}$ and $\al_{1}\left(\overline{h}\right),\beta_{1}\left(\overline{h}\right)\in \C\left[\,\overline{h}\,\right]$  with $\al_{1}\left(\overline{h}\right),\beta_{1}\left(\overline{h}\right)$ satisfying
\begin{align}\label{1}
\left(
\begin{array}{c}
p_{0} \\
p_{1} \\
p_{2} \\
\vdots \\
p_{m} \\
\end{array}
\right)=\lambda^{2}A\left(
                        \begin{array}{c}
                          q_{0} \\
                          q_{1} \\
                          q_{2} \\
                          \vdots \\
                          q_{m} \\
                        \end{array}
                      \right)
                      \ \ \ {\rm with}\ \ A=\left(
                       \begin{array}{ccccc}
                         1 & 2b & 2b^{2} & \cdots & 2b^{m} \\
                         0 & 1 & 2b & \cdots & 2b^{m-1} \\
                         0 & 0 & 1 &\cdots  & 2b^{m-2} \\
                         \vdots & \vdots & \vdots& \ddots & \vdots \\
                         0 & 0 & 0 & \cdots & 1 \\
                       \end{array}
                     \right),\end{align}
where $m={\rm deg}_{\overline{h}}\left(\al_{1}\left(\overline{h}\right)\right)={\rm deg}_{\overline{h}}\left(\beta_{1}\left(\overline{h}\right)\right)$ and $p_{i},q_{i}$, for $i\in \{0,\ldots,m\}$, are coefficients of $\al_{1}\left(\overline{h}\right)$ and $\beta_{1}\left(\overline{h}\right)$, respectively.\end{defi}
Indeed, using involution $\tau$, which is defined as below
\begin{eqnarray*}
  \tau: &e&\longmapsto -f, \ \ \overline{e}\longmapsto -\overline{f}, \ \ \ h\longmapsto -h,\ \ \
  f\longmapsto -e, \ \ \ \overline{f}\longmapsto -\overline{e},\ \ \ \overline{h}\longmapsto -\overline{h},
\end{eqnarray*}
we only need to verify that $\Gamma(\lambda,a,b)$ and $\Omega\left(\lambda,b,\beta_{1}\left(\overline{h}\right)\right)$ are $\g$-modules. First, we prove that $\Gamma(\lambda,a,b)$ is a $\g$-module. Let $\gamma\left(h,\overline{h}\right)\in \Gamma(\lambda,a,b)$. Then we have:
\begin{align*}
  \overline{e}\cdot f\cdot \gamma\left(h,\overline{h}\right) &=-\frac{1}{2\lambda}\overline{e}\cdot \left(\left((h+2)\overline{h}+b\right)\gamma\left(h+2,\overline{h}\right)
 +\left(\overline{h}^{2}+a\right)\overline{\partial}\left(\gamma\left(h+2,\overline{h}\right)\right)\right)  \\
 &=-\frac{1}{2}\left(h\overline{h}+b\right)\gamma\left(h,\overline{h}\right)
   -\frac{1}{2}\left(\overline{h}^{2}+a\right)\overline{\partial}\left(\gamma\left(h,\overline{h}\right)\right),  \\
 f\cdot \overline{e}\cdot \gamma\left(h,\overline{h}\right)&=\lambda f\cdot\gamma\left(h-2,\overline{h}\right)
 =-\frac{1}{2}\left((h+2)\overline{h}+b\right)\gamma\left(h,\overline{h}\right)
   -\frac{1}{2}\left(\overline{h}^{2}+a\right)\overline{\partial}\left(\gamma\left(h,\overline{h}\right)\right),
 \end{align*} which implies that $\overline{e}\cdot f\cdot \gamma\left(h,\overline{h}\right)-f\cdot\overline{e}\cdot \gamma\left(h,\overline{h}\right) =\overline{h}\cdot\gamma\left(h,\overline{h}\right)$. Next, from the actions of $e,\overline{e}$ and $\overline{f}$ on $\gamma\left(h,\overline{h}\right)$, we get
 \begin{align*}
 \overline{e}\cdot  \overline{f}\cdot\gamma\left(h,\overline{h}\right) &=-\frac{1}{4\lambda} \overline{e}\cdot \left( \left(\overline{h}^{2}+a\right)\gamma\left(h+2,\overline{h}\right)\right)
 =-\frac{1}{4}\left(\overline{h}^{2}+a\right)\gamma\left(h,\overline{h}\right),\\
 \overline{f}\cdot  \overline{e}\cdot\gamma\left(h,\overline{h}\right) &=\lambda \overline{f}\cdot \gamma\left(h-2,\overline{h}\right)=-\frac{1}{4}\left(\overline{h}^{2}+a\right)\gamma\left(h,\overline{h}\right),\\
  e\cdot \overline{e}\cdot \gamma\left(h,\overline{h}\right) &=\lambda e\cdot \gamma\left(h-2,\overline{h}\right)=-2\lambda^{2}\overline{\partial}\left(\gamma\left(h-4,\overline{h}\right)\right),  \\
  \overline{e}\cdot e\cdot \gamma\left(h,\overline{h}\right)&=-2\lambda\overline{e}\cdot \overline{\partial}\left(\gamma\left(h-2,\overline{h}\right)\right)=-2\lambda^{2}\overline{\partial}\left(\gamma\left(h-4,\overline{h}\right)\right),
\end{align*}
which shows that $\overline{e}\cdot\overline{f}\cdot \gamma\left(h,\overline{h}\right)-\overline{f}\cdot\overline{e}\cdot \gamma\left(h,\overline{h}\right)= e\cdot\overline{e}\cdot \gamma\left(h,\overline{h}\right)-\overline{e}\cdot e\cdot \gamma\left(h,\overline{h}\right) = 0$. Moreover, from a little lengthy but straightforward computation, we obtain
\begin{align*}
 e\cdot f\cdot \gamma\left(h,\overline{h}\right)
 &=h\gamma\left(h,\overline{h}\right)+\left((h+2)\overline{h}+b\right)\overline{\partial}\left(\gamma\left(h,\overline{h}\right)\right)
        +\left(\overline{h}^{2}+a\right)\overline{\partial}^{2}\left(\gamma\left(h,\overline{h}\right)\right),\\
 f\cdot e\cdot \gamma\left(h,\overline{h}\right)
 &=\left((h+2)\overline{h}+b\right)\overline{\partial}\left(\gamma\left(h,\overline{h}\right)\right)+\left(\overline{h}^{2}+a\right)
       \overline{\partial}^{2}\left(\gamma\left(h,\overline{h}\right)\right),  \\
f\cdot \overline{f}\cdot \gamma\left(h,\overline{h}\right)
&=\frac{1}{8\lambda^{2}}\left(\overline{h}^{2}+a\right)\left(\left((h+4)\overline{h}+b\right)\gamma\left(h+4,\overline{h}\right)
+\left(\overline{h}^{2}+a\right)\overline{\partial}\left(\gamma\left(h+4,\overline{h}\right)\right)\right),  \\
 \overline{f}\cdot f\cdot \gamma\left(h,\overline{h}\right)
   &=\frac{1}{8\lambda^{2}}\left(\overline{h}^{2}+a\right)\left(\left((h+4)\overline{h}+b\right)\gamma\left(h+4,\overline{h}\right)
+\left(\overline{h}^{2}+a\right)\overline{\partial}\left(\gamma\left(h+4,\overline{h}\right)\right)\right),\\
e\cdot \overline{f}\cdot \gamma\left(h,\overline{h}\right) &=
\overline{h}\gamma\left(h,\overline{h}\right)+\frac{\overline{h}^{2}+a}{2}\,\overline{\partial}\left(\gamma\left(h,\overline{h}\right)\right),\ \ \ \ \overline{f}\cdot e\cdot \gamma\left(h,\overline{h}\right) =
\frac{\overline{h}^{2}+a}{2}\,\overline{\partial}\left(\gamma\left(h,\overline{h}\right)\right),
\end{align*}
which yields the following equations
\begin{align*}
  e\cdot f\cdot\gamma\left(h,\overline{h}\right)-f\cdot e \cdot\gamma\left(h,\overline{h}\right) &=h\cdot\gamma\left(h,\overline{h}\right), \ \ \ \ \ f\cdot\overline{f}\cdot\gamma\left(h,\overline{h}\right)-\overline{f}\cdot f\cdot\gamma\left(h,\overline{h}\right)= 0,\\
  e\cdot\overline{f}\cdot\gamma\left(h,\overline{h}\right)-\overline{f}\cdot e\cdot\gamma\left(h,\overline{h}\right) &=\overline{h}\cdot\gamma\left(h,\overline{h}\right).
\end{align*}
Thus, $\Gamma(\lambda,a,b)$ is a $\g$-module.

Next, we want to show that the actions of $e$, $f$, $\overline{e}$ and $\overline{f}$ on $\Omega\left(\lambda,b,\beta_{1}\left(\overline{h}\right)\right)$ satisfy the relations in $\g$. According to the actions of $e,f$ on $\Omega\left(\lambda,b,\beta_{1}\left(\overline{h}\right)\right)$, respectively, for any $\gamma\left(h,\overline{h}\right)\in \Omega\left(\lambda,b,\beta_{1}\left(\overline{h}\right)\right)$, we have
\begin{eqnarray*}\!\!\!\!\!\!\!\!\!\!\!\!&\!\!\!\!\!\!\!\!\!\!\!\!\!\!\!&
\ \ \ \ \ \ e\cdot f\cdot \gamma\left(h,\overline{h}\right)=-\frac{\lambda h+2\al_{1}\left(\overline{h}\right)}{2}\left(\frac{h-2}{2\lambda}-\beta_{1}\left(\overline{h}\right)\right)\gamma\left(h,\overline{h}\right)
   -\lambda\left(\overline{h}+b\right)\overline{\partial}\left(\beta_{1}\left(\overline{h}\right)\right)\gamma\left(h,\overline{h}\right)
\nonumber\\\!\!\!\!\!\!\!\!\!\!\!\!&\!\!\!\!\!\!\!\!\!\!\!\!\!\!\!&
 \ \ \ \ \ \ \ \ \ \ \ \ \ \ \ \ \ \ \ \ \ \ \ \ \ \ \
+\kappa_{1}\left(h,\overline{h}\right)\overline{\partial}\left(\gamma\left(h,\overline{h}\right)\right)
 +\left(\overline{h}+b\right)\left(\overline{h}-b\right)\overline{\partial}^{2}\left(\gamma\left(h,\overline{h}\right)\right),
\nonumber\\\!\!\!\!\!\!\!\!\!\!\!\!&\!\!\!\!\!\!\!\!\!\!\!\!\!\!\!&
\ \ \ \ \ \ f\cdot e\cdot \gamma\left(h,\overline{h}\right)=-\frac{\lambda (h+2)+2\al_{1}\left(\overline{h}\right)}{2}\left(\frac{h}{2\lambda}-\beta_{1}\left(\overline{h}\right)\right)\gamma\left(h,\overline{h}\right)
   -\frac{\overline{h}-b}{\lambda}\overline{\partial}\left(\al_{1}\left(\overline{h}\right)\right)\gamma\left(h,\overline{h}\right)
\nonumber\\\!\!\!\!\!\!\!\!\!\!\!\!&\!\!\!\!\!\!\!\!\!\!\!\!\!\!\!&
 \ \ \ \ \ \ \ \ \ \ \ \ \ \ \ \ \ \ \ \ \ \ \ \ \ \ \ \ +\kappa_{1}\left(h,\overline{h}\right)\overline{\partial}\left(\gamma\left(h,\overline{h}\right)\right)
 +\left(\overline{h}+b\right)\left(\overline{h}-b\right)\overline{\partial}^{2}\left(\gamma\left(h,\overline{h}\right)\right),
 \end{eqnarray*}
 where $\kappa_{1}\left(h,\overline{h}\right)=\lambda\left(\overline{h}+b\right)\left(\frac{h}{2\lambda}-
\b_{1}\left(\overline{h}\right)\right)-\frac{\overline{h}-b}{\lambda}\left(\frac{\lambda h}{2}+\al_{1}\left(\overline{h}\right)\right)$. Then, combined with (\ref{1}), it follows that
\begin{eqnarray*}\!\!\!\!\!\!\!\!\!\!\!\!&\!\!\!\!\!\!\!\!\!\!\!\!\!\!\!&
  e\cdot f\cdot \gamma\left(h,\overline{h}\right)-f\cdot e\cdot \gamma\left(h,\overline{h}\right)=h\gamma\left(h,\overline{h}\right)+\frac{1}{\lambda}\left(\al_{1}\left(\overline{h}\right)
  -\lambda^{2}\beta_{1}\left(\overline{h}\right)\right)\gamma\left(h,\overline{h}\right)
\nonumber\\\!\!\!\!\!\!\!\!\!\!\!\!&\!\!\!\!\!\!\!\!\!\!\!\!\!\!\!&
 \ \ \ -\frac{1}{\lambda}\left(\lambda^{2}\left(\overline{h}+b\right)\overline{\partial}\left(\beta_{1}\left(\overline{h}\right)\right)
-\left(\overline{h}-b\right)\overline{\partial}\left(\al_{1}\left(\overline{h}\right)\right) \right)\gamma\left(h,\overline{h}\right)
   =h\cdot\gamma\left(h,\overline{h}\right).
\end{eqnarray*}
Next, considering the action of $\overline{f}$ on $\gamma\left(h,\overline{h}\right)$, we have
\begin{eqnarray*}\!\!\!\!\!\!\!\!\!\!\!\!&\!\!\!\!\!\!\!\!\!\!\!\!\!\!\!&
\ \ \ \ \ \ e\cdot \overline{f}\cdot \gamma\left(h,\overline{h}\right)
=-\frac{1}{2\lambda}\left(\overline{h}-b\right)\left(\frac{\lambda}{2}h+\al_{1}\left(\overline{h}\right)\right)\gamma\left(h,\overline{h}\right)
+\frac{1}{2}\left(\overline{h}+b\right)\gamma\left(h,\overline{h}\right)
+\frac{1}{2}\kappa_{2}\left(\overline{h}\right)\overline{\partial}\left(\gamma\left(h,\overline{h}\right)\right),
\nonumber\\\!\!\!\!\!\!\!\!\!\!\!\!&\!\!\!\!\!\!\!\!\!\!\!\!\!\!\!&
\ \ \ \ \ \ \overline{f}\cdot e\cdot \gamma\left(h,\overline{h}\right)=
-\frac{1}{2\lambda}\left(\overline{h}-b\right)\left(\frac{\lambda}{2}(h+2)+\al_{1}\left(\overline{h}\right)\right)\gamma\left(h,\overline{h}\right)
+\frac{1}{2}\kappa_{2}\left(\overline{h}\right)\overline{\partial}\left(\gamma\left(h,\overline{h}\right)\right),
\nonumber\\\!\!\!\!\!\!\!\!\!\!\!\!&\!\!\!\!\!\!\!\!\!\!\!\!\!\!\!&
\ \ \ \ \ \ f\cdot \overline{f}\cdot \gamma\left(h,\overline{h}\right)=\frac{1}{2\lambda}\left(\overline{h}-b\right)\left(\frac{h+2}{2\lambda}-\beta_{1}\left(\overline{h}\right)\right)\gamma\left(h+4,\overline{h}\right)
+\frac{\left(\overline{h}-b\right)^{2}}{2\lambda^{2}}\overline{\partial}\left(\gamma\left(h+4,\overline{h}\right)\right) ,
\nonumber\\\!\!\!\!\!\!\!\!\!\!\!\!&\!\!\!\!\!\!\!\!\!\!\!\!\!\!\!&
\ \ \ \ \ \ \overline{f}\cdot f\cdot \gamma\left(h,\overline{h}\right)=\frac{1}{2\lambda}\left(\overline{h}-b\right)\left(\frac{h+2}{2\lambda}-\beta_{1}\left(\overline{h}\right)\right)\gamma\left(h+4,\overline{h}\right)
+\frac{\left(\overline{h}-b\right)^{2}}{2\lambda^{2}}\overline{\partial}\left(\gamma\left(h+4,\overline{h}\right)\right),
\end{eqnarray*}
where $\kappa_{2}\left(\overline{h}\right)=\left(\overline{h}+b\right)\left(\overline{h}-b\right)$. Then it follows that $ e\cdot \overline{f}\cdot \gamma\left(h,\overline{h}\right)-\overline{f}\cdot e\cdot \gamma\left(h,\overline{h}\right)=\overline{h}\cdot\gamma\left(h,\overline{h}\right)$ and $f\cdot \overline{f}\cdot \gamma\left(h,\overline{h}\right)-\overline{f}\cdot f\cdot \gamma\left(h,\overline{h}\right)=0$.
Furthermore, by a little lengthy but straightforward computation, we obtain
\begin{eqnarray*}\!\!\!\!\!\!\!\!\!\!\!\!&\!\!\!\!\!\!\!\!\!\!\!\!\!\!\!&
\ \ \ \ \ \ \ \ \ \ \ \ \ \overline{e}\cdot f\cdot \gamma\left(h,\overline{h}\right)
=-\frac{\lambda}{2}\left(\overline{h}+b\right)\left(\frac{h-2}{2\lambda}-\beta_{1}\left(\overline{h}\right)\right)\gamma\left(h,\overline{h}\right)
-\frac{1}{2}\kappa_{2}\left(\overline{h}\right)\overline{\partial}\left(\gamma\left(h,\overline{h}\right)\right),
\nonumber\\\!\!\!\!\!\!\!\!\!\!\!\!&\!\!\!\!\!\!\!\!\!\!\!\!\!\!\!&
\ \ \ \ \ \ \ \ \ \ \ \ \ f\cdot \overline{e}\cdot \gamma\left(h,\overline{h}\right)=
-\frac{\lambda}{2}\left(\overline{h}+b\right)\left(\frac{h}{2\lambda}-\beta_{1}\left(\overline{h}\right)\right)\gamma\left(h,\overline{h}\right)
-\frac{1}{2}\left(\overline{h}-b\right)\gamma\left(h,\overline{h}\right)
-\frac{1}{2}\kappa_{2}\left(\overline{h}\right)\overline{\partial}\left(\gamma\left(h,\overline{h}\right)\right),
\nonumber\\\!\!\!\!\!\!\!\!\!\!\!\!&\!\!\!\!\!\!\!\!\!\!\!\!\!\!\!&
\ \ \ \ \ \ \ \ \ \ \ \ \ e\cdot \overline{e}\cdot \gamma\left(h,\overline{h}\right)=
\frac{\lambda}{2}\left(\overline{h}+b\right)\left(\frac{\lambda (h-2)}{2}+\al_{1}\left(\overline{h}\right)\right)\gamma\left(h-4,\overline{h}\right)
-\frac{\lambda^{2}}{2}\left(\overline{h}+b\right)^{2}\overline{\partial}\left(\gamma\left(h-4,\overline{h}\right)\right),
\nonumber\\\!\!\!\!\!\!\!\!\!\!\!\!&\!\!\!\!\!\!\!\!\!\!\!\!\!\!\!&
\ \ \ \ \ \ \ \ \ \ \ \ \ \overline{e}\cdot e\cdot \gamma\left(h,\overline{h}\right)=
\frac{\lambda}{2}\left(\overline{h}+b\right)\left(\frac{\lambda(h-2)}{2}+\al_{1}\left(\overline{h}\right)\right)\gamma\left(h-4,\overline{h}\right)
-\frac{\lambda^{2}}{2}\left(\overline{h}+b\right)^{2}\overline{\partial}\left(\gamma\left(h-4,\overline{h}\right)\right),
\nonumber\\\!\!\!\!\!\!\!\!\!\!\!\!&\!\!\!\!\!\!\!\!\!\!\!\!\!\!\!&
\ \ \ \ \ \ \ \ \ \ \ \ \ \overline{e}\cdot \overline{f}\cdot \gamma\left(h,\overline{h}\right)= -\frac{1}{4}\kappa_{2}\left(\overline{h}\right)\gamma\left(h,\overline{h}\right),
\ \ \ \ \ \ \ \ \ \overline{f}\cdot \overline{e}\cdot \gamma\left(h,\overline{h}\right)=-\frac{1}{4}\kappa_{2}\left(\overline{h}\right)\gamma\left(h,\overline{h}\right),
\end{eqnarray*}
where $\kappa_{2}\left(\overline{h}\right)$ is defined as above. Then we obtain \begin{align*}
e\cdot \overline{e}\cdot \gamma\left(h,\overline{h}\right)-\overline{e}\cdot e\cdot \gamma\left(h,\overline{h}\right) &= 0, \ \ \ \ \,
\overline{e}\cdot f\cdot \gamma\left(h,\overline{h}\right)-f\cdot \overline{e}\cdot \gamma\left(h,\overline{h}\right)=\overline{h}\cdot \gamma\left(h,\overline{h}\right),\\
\overline{e}\cdot \overline{f}\cdot \gamma\left(h,\overline{h}\right)-\overline{f}\cdot \overline{e}\cdot \gamma\left(h,\overline{h}\right) &=0.
\end{align*}
Consequently, we can see that $\Omega\left(\lambda,b,\beta_{1}\left(\overline{h}\right)\right)$ is a $\g$-module.

\subsection{Classification of free $U\left(\overline{\mathfrak{h}}\right)$-modules of rank 1}
In this section, we will prove Theorems \ref{prop-3.2bc} and \ref{prop-3.2a}. First, we give the proof of Theorem \ref{prop-3.2bc}, which gives a complete classification of free $U\left(\overline{\mathfrak{h}}\right)$-modules of rank 1 over Takiff $\mathfrak{sl}_{2}$ $\g$.
\subsubsection{Proof of Theorem \ref{prop-3.2bc}}
Let $M$ be a $\g$-module such that it is free of rank 1 as a $U\left(\overline{\mathfrak{h}}\right)$-module, where $U\left(\overline{\mathfrak{h}}\right)=\C\left[h,\overline{h}\,\right]$. Take $1\in M$, then we can write $M=U\left(\overline{\mathfrak{h}}\right)\cdot1=\C\left[h,\overline{h}\,\right]\cdot1$. Here and below, we denote $e:=e\otimes 1$, $f:=f\otimes 1$, $h:=h\otimes 1$ and write
\begin{align*}
  e\cdot 1= \al\left(h,\overline{h}\right),\ \ \  f\cdot 1= \beta\left(h,\overline{h}\right),\ \ \ \
  \overline{e}\cdot 1=  \overline{\al}\left(h,\overline{h}\right), \ \ \ \overline{f}\cdot 1= \overline{\beta}\left(h,\overline{h}\right).
\end{align*}
Then the proof of Theorem \ref{prop-3.2bc} follows from Lemma \ref{lemma1.1} combined with the following two Lemmas.
\begin{lemm}${\rm deg}_{h}\overline{\al}\left(h,\overline{h}\right)={\rm deg}_{h}\overline{\beta}\left(h,\overline{h}\right)=0.$\end{lemm}
\begin{proof}Due to Lemma \ref{lemma1.1}, applying $\overline{e}$ on $\overline{f}\cdot1$ and $\overline{f}$ on $\overline{e}\cdot 1$ respectively, we obtain
\begin{align*}
  \overline{e}\cdot\overline{f}\cdot1 &=\overline{e}\cdot \overline{\beta}\left(h,\overline{h}\right) = \overline{\beta}\left(h-2,\overline{h}\right)\overline{\al}\left(h,\overline{h}\right),  \\
  \overline{f}\cdot \overline{e}\cdot 1&=\overline{f}\cdot \overline{\al}\left(h,\overline{h}\right) = \overline{\al}\left(h+2,\overline{h}\right)\overline{\beta}\left(h,\overline{h}\right).
\end{align*}
This implies that
\begin{equation}\label{eq90}
0=\overline{\beta}\left(h-2,\overline{h}\right)\overline{\al}\left(h,\overline{h}\right)
-\overline{\al}\left(h+2,\overline{h}\right)\overline{\beta}\left(h,\overline{h}\right).
\end{equation}
Now we write $\overline{\al}\left(h,\overline{h}\right)=\sum_{i=0}^{m}c_{i}h^{i}\kappa_{i}\left(\overline{h}\right)$ and $\overline{\beta}\left(h,\overline{h}\right)=\sum_{j=0}^{n}d_{j}h^{j}\kappa'_{j}\left(\overline{h}\right)$, for some $m,n\in \Z_{+}$, $c_{i},d_{j}\in \C$ and $\kappa_{i}\left(\overline{h}\right), \kappa'_{j}\left(\overline{h}\right)\in \C\left[\,\overline{h}\,\right]$. Then, comparing the coefficients of $h^{m+n-1}\kappa_{m}\left(\overline{h}\right)\kappa'_{n}\left(\overline{h}\right)$  in equation (\ref{eq90}), we have
\begin{equation*}
  m=n=0.
\end{equation*}
Thus it follows that ${\rm deg}_{h}\overline{\al}\left(h,\overline{h}\right)={\rm deg}_{h}\overline{\beta}\left(h,\overline{h}\right)=0$ and the proof is completed.
\end{proof}
For convenience, we write $\overline{e}\cdot 1=\overline{\al}\left(\overline{h}\right)$, $\overline{f}\cdot 1=\overline{\beta}\left(\overline{h}\right)$ in the following.
\begin{lemm}
\begin{itemize}
  \item[\rm(1)]If $e\cdot 1=0$ and $f\cdot 1\neq0$, then we have
  \begin{equation}\label{eqi1}
    {\rm deg}_{h}\beta\left(h,\overline{h}\right)={\rm deg}_{\overline{h}}\beta\left(h,\overline{h}\right)=1, \ \ {\rm deg}_{\overline{h}}\overline{\al}\left(\overline{h}\right)=0,\ \ {\rm deg}_{\overline{h}}\overline{\beta}\left(\overline{h}\right)=2.
  \end{equation}
  Consequently, $\beta\left(h,\overline{h}\right)=-\frac{1}{2\lambda}\left((h+2)\overline{h}+b\right)$, $\overline{\al}\left(\overline{h}\right)=\lambda$ and $\overline{\beta}\left(\overline{h}\right)=-\frac{1}{4\lambda}\left(\overline{h}^{2}+a\right)$, where $\lambda\in \C^{\times}$ and $a,b\in \C$.
  \item[\rm(2)]If $e\cdot 1\neq0$ and $f\cdot 1=0$, then we have
  \begin{equation}\label{eqi2}
    {\rm deg}_{h}\al\left(h,\overline{h}\right)={\rm deg}_{\overline{h}}\al\left(h,\overline{h}\right)=1, \ \ {\rm deg}_{\overline{h}}\overline{\al}\left(\overline{h}\right)=2, \ \ {\rm deg}_{\overline{h}}\overline{\beta}\left(\overline{h}\right)=0.
  \end{equation}
  Thus, $\al\left(h,\overline{h}\right)=-\frac{1}{2\lambda}\left((h-2)\overline{h}+b\right)$, $\overline{\al}\left(\overline{h}\right)=-\frac{1}{4\lambda}\left(\overline{h}^{2}+a\right)$ and $\overline{\beta}\left(\overline{h}\right)=\lambda$, where $\lambda\in \C^{\times}$ and $a,b\in \C$.
  \item[\rm(3)]If $e\cdot 1\neq0$ and $f\cdot 1\neq0$, then we have
  \begin{align*}
  {\rm deg}_{h}\al\left(h,\overline{h}\right)=1,
  \ \ \ {\rm deg}_{h}\beta\left(h,\overline{h}\right)=1,\ \ \ \
{\rm deg}_{\overline{h}}\overline{\al}\left(\overline{h}\right)=1,
  \ \ \ {\rm deg}_{\overline{h}}\overline{\beta}\left(\overline{h}\right)=1.
\end{align*}
Therefore, we have $\al\left(h,\overline{h}\right)=\frac{\lambda}{2}h+\al_{1}\left(\overline{h}\right)$, $ \beta\left(h,\overline{h}\right)=-\frac{1}{2\lambda}h+\beta_{1}\left(\overline{h}\right)$, $\overline{\al}\left(\overline{h}\right)=\frac{\lambda}{2}\left(\overline{h}+b\right)$ and $\overline{\beta}\left(\overline{h}\right)=-\frac{1}{2\lambda}\left(\overline{h}-b\right)$, for some $\lambda\in \C^{\times}$, $b\in \C$ and $\al_{1}\left(\overline{h}\right),\beta_{1}\left(\overline{h}\right)\in \C \left[\,\overline{h}\,\right]$, with $\al_{1}\left(\overline{h}\right),\beta_{1}\left(\overline{h}\right)$ satisfying
$\left(
\begin{array}{c}
p_{0} \\
p_{1} \\
p_{2} \\
\vdots \\
p_{m} \\
\end{array}
\right)=\lambda^{2}A\left(
                        \begin{array}{c}
                          q_{0} \\
                          q_{1} \\
                          q_{2} \\
                          \vdots \\
                          q_{m} \\
                        \end{array}
                      \right)$ with $A=\left(
                       \begin{array}{ccccc}
                         1 & 2b & 2b^{2} & \cdots & 2b^{m} \\
                         0 & 1 & 2b & \cdots & 2b^{m-1} \\
                         0 & 0 & 1 &\cdots  & 2b^{m-2} \\
                         \vdots & \vdots & \vdots& \ddots & \vdots \\
                         0 & 0 & 0 & \cdots & 1 \\
                       \end{array}
                     \right)$, where $m={\rm deg}_{\overline{h}}\left(\al_{1}\left(\overline{h}\right)\right)={\rm deg}_{\overline{h}}\left(\beta_{1}\left(\overline{h}\right)\right)$ and $p_{i},q_{i}\in \C$, for $i\in \{0,\ldots,m\}$, are the coefficients of $\al_{1}\left(\overline{h}\right)$ and $\beta_{1}\left(\overline{h}\right)$, respectively.
\end{itemize}
\end{lemm}
\begin{proof}From Lemma \ref{lemma1.1} and the relations in $\g$, we have
\begin{eqnarray}&\!\!\!\!\!\!\!\!\!\!\!\!\!\!\!\!\!\!\!\!\!\!\!\!&
  \overline{h}=[\overline{e},f]\cdot 1=\overline{e}\cdot f\cdot 1-f\cdot\overline{e}\cdot 1=\overline{e}\cdot \beta\left(h,\overline{h}\right)-f\cdot \overline{\al}\left(\overline{h}\right)
\nonumber\\\!\!\!\!\!\!\!\!\!\!\!\!&\!\!\!\!\!\!\!\!\!\!\!\!\!\!\!&
\ \ \, =\left(\beta\left(h-2,\overline{h}\right)-\beta\left(h,\overline{h}\right)\right)\overline{\al}\left(\overline{h}\right)
-2\overline{\partial}\left(\overline{\al}\left(\overline{h}\right)\right)\overline{\beta}\left(\overline{h}\right),\label{eq2.1}\\
 &\!\!\!\!\!\!\!\!\!\!\!\!\!\!\!\!\!\!\!\!\!\!\!\!&
 \overline{h}=\left[\,e,\overline{f}\,\right]\cdot 1=e\cdot\overline{f}\cdot 1-\overline{f}\cdot e\cdot 1=e\cdot \overline{\beta}\left(\overline{h}\right)-\overline{f}\cdot \al\left(h,\overline{h}\right)
 \nonumber\\\!\!\!\!\!\!\!\!\!\!\!\!&\!\!\!\!\!\!\!\!\!\!\!\!\!\!\!&
 \ \ \, =\left(\al\left(h,\overline{h}\right)-\al\left(h+2,\overline{h}\right)\right)\overline{\beta}\left(\overline{h}\right)
 -2\overline{\partial}\left(\overline{\beta}\left(\overline{h}\right)\right)\overline{\al}\left(\overline{h}\right),\label{eq2.2} \\
  &\!\!\!\!\!\!\!\!\!\!\!\!\!\!\!\!\!\!\!\!\!\!\!\!&
  h=[e,f]\cdot 1=e\cdot f\cdot 1 -f\cdot e\cdot 1=e\cdot \beta\left(h,\overline{h}\right)-f\cdot \al\left(h,\overline{h}\right)
   \nonumber\\\!\!\!\!\!\!\!\!\!\!\!\!&\!\!\!\!\!\!\!\!\!\!\!\!\!\!\!&
\ \ \,=\beta\left(h-2,\overline{h}\right)\al\left(h,\overline{h}\right)-\al\left(h+2,\overline{h}\right)\beta\left(h,\overline{h}\right)
   \nonumber\\\!\!\!\!\!\!\!\!\!\!\!\!&\!\!\!\!\!\!\!\!\!\!\!\!\!\!\!&
 \ \ \ \ \ \ -2\overline{\partial}\left(\beta\left(h-2,\overline{h}\right)\right)\overline{\al}\left(\overline{h}\right)
  -2\overline{\partial}\left(\al\left(h+2,\overline{h}\right)\right)\overline{\beta}\left(\overline{h}\right),\label{eq2.3}\\
  &\!\!\!\!\!\!\!\!\!\!\!\!\!\!\!\!\!\!\!\!\!\!\!\!&
  0=[e,\overline{e}]\cdot 1=e\cdot \overline{e}\cdot 1 -\overline{e}\cdot e\cdot 1=e\cdot \overline{\al}\left(\overline{h}\right)-\overline{e}\cdot \al\left(h,\overline{h}\right)
   \nonumber\\\!\!\!\!\!\!\!\!\!\!\!\!&\!\!\!\!\!\!\!\!\!\!\!\!\!\!\!&
 \ \ \, =\overline{\al}\left(\overline{h}\right)\left(\al\left(h,\overline{h}\right)-\al\left(h-2,\overline{h}\right)
 -2\overline{\partial}\left(\overline{\al}\left(\overline{h}\right)\right)\right), \label{eqi2.4}\\
  &\!\!\!\!\!\!\!\!\!\!\!\!\!\!\!\!\!\!\!\!\!\!\!\!&
  0=\left[f,\overline{f}\,\right]\cdot 1=f\cdot \overline{f}\cdot 1 -\overline{f}\cdot f\cdot 1=f\cdot \overline{\beta}\left(\overline{h}\right)-\overline{f}\cdot \beta\left(h,\overline{h}\right)
  \nonumber\\\!\!\!\!\!\!\!\!\!\!\!\!&\!\!\!\!\!\!\!\!\!\!\!\!\!\!\!&
 \ \ \, =\overline{\beta}\left(\overline{h}\right)\left(\beta\left(h,\overline{h}\right)-\beta\left(h+2,\overline{h}\right)
      +2\overline{\partial}\left(\overline{\beta}\left(\overline{h}\right)\right)\right).\label{eqi2.5}
\end{eqnarray}

First, we consider the case $e\cdot 1=0$ and $f\cdot 1\neq0$. From (\ref{eq2.2}) and (\ref{eq2.3}), we obtain \begin{equation*}
        \overline{h}=-2\overline{\partial}\left(\overline{\beta}\left(\overline{h}\right)\right)\overline{\al}\left(\overline{h}\right),\ \ \
        h=-2\overline{\partial}\left(\beta\left(h-2,\overline{h}\right)\right)\overline{\al}\left(\overline{h}\right).
        \end{equation*}
Comparing the degree of $\overline{h}$ and $h$, respectively, in the above two equations, we see that ${\rm deg}_{h}\beta\left(h,\overline{h}\right)={\rm deg}_{\overline{h}}\beta\left(h,\overline{h}\right)=1,\ {\rm deg}_{\overline{h}}\overline{\al}\left(\overline{h}\right)=0,\ {\rm deg}_{\overline{h}}\overline{\beta}\left(\overline{h}\right)=2$.

Writing $\beta\left(h,\overline{h}\right)=d_{0,0}+d_{0,1}h+\left(d_{1,0}+d_{1,1}h\right)\overline{h}$, $\overline{\al}\left(\overline{h}\right)=\lambda$ and $\overline{\beta}\left(\overline{h}\right)=b_{2}\overline{h}^{2}+b_{1}\overline{h}+b_{0}$, for some $\lambda,b_{2}\in \C^{\times}$, $d_{1,i},d_{0,i},b_{i}\in \C$ with $i=0,1$ and $(d_{1,0}+d_{1,1}h)\left(d_{0,1}+d_{1,1}\overline{h}\right)\neq 0$, by (\ref{eq2.1})-(\ref{eqi2.5}), we have
\begin{eqnarray*}
\overline{h}&=&-2\lambda \left(d_{0,1}+d_{1,1}\overline{h}\right)=-2\lambda\left(2b_{2}\overline{h}+b_{1}\right),\ \ \ \ \ h=-2\lambda \left(d_{1,0}-2d_{1,1}+d_{1,1}h\right),\\
0&=&\left(b_{2}\overline{h}^{2}+b_{1}\overline{h}+b_{0}\right)\left(2b_{2}\overline{h}+b_{1}-d_{0,1}-d_{1,1}\overline{h}\right).
\end{eqnarray*}
This implies that $d_{0,1}=b_{1}=d_{1,0}-2d_{1,1}=0$ and $-2\lambda d_{1,1}=-4\lambda b_{2}=1$. Thus, we have $\beta\left(h,\overline{h}\right)=-\frac{1}{2\lambda}\left((h+2)\overline{h}+b\right)$, $\overline{\al}\left(\overline{h}\right)=\lambda$ and $\overline{\beta}\left(\overline{h}\right)=-\frac{1}{4\lambda}\left(\overline{h}^{2}+a\right)$, where $\lambda\in \C^{\times}$ and $b:=-2\lambda d_{0,0}$, $a:=-4\lambda b_{0}\in \C$.

Next, consider the case $e\cdot 1\neq0$ and $f\cdot 1=0$. By symmetry, we have ${\rm deg}_{h}\al\left(h,\overline{h}\right)={\rm deg}_{\overline{h}}\al\left(h,\overline{h}\right)=1,\ {\rm deg}_{\overline{h}}\overline{\al}\left(\overline{h}\right)=2,\ {\rm deg}_{\overline{h}}\overline{\beta}\left(\overline{h}\right)=0$. Applying the involution
\begin{equation*}
  \tau: e\mapsto -f,\ \ f\mapsto -e,\ \ h\mapsto -h,\ \ \overline{e}\mapsto -\overline{f},\ \ \overline{f}\mapsto -\overline{e}, \ \ \overline{h}\mapsto -\overline{h},
\end{equation*}
we get $\al\left(h,\overline{h}\right)=-\frac{1}{2\lambda}\left((h-2)\overline{h}+b\right)$, $\overline{\al}\left(\overline{h}\right)=-\frac{1}{4\lambda}\left(\overline{h}^{2}+a\right)$ and $\overline{\beta}\left(\overline{h}\right)=\lambda$, for some $\lambda\in \C^{\times}$ and $a,b\in \C$.

Now we consider the case  $e\cdot 1\neq0$ and $f\cdot 1\neq0$. Assume ${\rm deg}_{\overline{h}}\overline{\al}\left(\overline{h}\right)=r_{\overline{\al}}$, ${\rm deg}_{\overline{h}}\overline{\beta}\left(\overline{h}\right)=t_{\overline{\beta}}$, ${\rm deg}_{\overline{h}}\al\left(h,\overline{h}\right)=n_{\al}$ and ${\rm deg}_{\overline{h}}\beta\left(h,\overline{h}\right)=n_{\beta}$. Observing the  degree of $h$ in (\ref{eq2.1})-(\ref{eq2.3}), we get
\begin{equation*}
  {\rm deg}_{h}\al\left(h,\overline{h}\right)\leq 1, \ \ \ {\rm deg}_{h}\beta\left(h,\overline{h}\right)\leq 1,\ \ {\rm and}\ \ {\rm deg}_{h}\al\left(h,\overline{h}\right)+{\rm deg}_{h}\beta\left(h,\overline{h}\right)> 0.
\end{equation*}
Now we consider the following cases.

\noindent{\bf Case 1:} ${\rm deg}_{h}\al\left(h,\overline{h}\right)+{\rm deg}_{h}\beta\left(h,\overline{h}\right)=1$. By symmetry, we only consider the case ${\rm deg}_{h}\al\left(h,\overline{h}\right)=1$, ${\rm deg}_{h}\beta\left(h,\overline{h}\right)=0$. Comparing the degree of $\overline{h}$ in (\ref{eq2.1}), (\ref{eq2.2}) and (\ref{eqi2.5}), respectively, we obtain
\begin{equation*}
 r_{\overline{\al}}=2,\ \ t_{\overline{\beta}}=0,\ \ \  n_{\al}=1.
\end{equation*}
Write $\overline{\beta}\left(\overline{h}\right)=\lambda$, $\overline{\al}\left(\overline{h}\right)=c_{2}\overline{h}^{2}+c_{1}\overline{h}+c_{0}$ and $\al\left(h,\overline{h}\right)=a_{0,0}+a_{0,1}h+\left(a_{1,0}+a_{1,1}h\right)\overline{h}$, where $\lambda, c_{2} \in \C^{\times}$, $c_{i},a_{0,i}, a_{1,i}\in \C$ with $i=0,1$ and $(a_{1,0}+a_{1,1}h)\left(a_{0,1}+a_{1,1}\overline{h}\right)\neq0$. Then, according to (\ref{eq2.1}) and (\ref{eq2.2}), we can have $a_{0,1}=c_{1}=0$ and $2c_{2}=a_{1,1}=-\frac{1}{2\lambda}$. Next, applying these in (\ref{eq2.3}) and comparing the degree of $\overline{h}$ on both sides, we obtain a contradiction. Thus, Case 1 does not occur.

\noindent{\bf Case 2:} ${\rm deg}_{h}\al\left(h,\overline{h}\right)={\rm deg}_{h}\beta\left(h,\overline{h}\right)=1$. In the following, we write $\overline{\al}\left(\overline{h}\right)=\sum_{i=0}^{r_{\overline{\al}}}a_{i}\overline{h}^{i}$, $\overline{\beta}\left(\overline{h}\right)=\sum_{j=0}^{t_{\overline{\beta}}}b_{j}\overline{h}^{j}$, $\al\left(h,\overline{h}\right)=\al_{1}\left(\overline{h}\right)+\al_{2}\left(\overline{h}\right)h$ and $\beta\left(h,\overline{h}\right)=\beta_{1}\left(\overline{h}\right)+\beta_{2}\left(\overline{h}\right)h$, where $\al_{1}\left(\overline{h}\right),0\neq\al_{2}\left(\overline{h}\right),\beta_{1}\left(\overline{h}\right),0\neq\beta_{2}\left(\overline{h}\right)\in \C\left[\,\overline{h}\,\right]$ with 
 $a_{i},b_{j}\in \C$ for $i\in\{0,1,\cdots, r_{\overline{\al}}-1\}$, $j\in\{0,1,\cdots, t_{\overline{\beta}}-1\}$ and $a_{r_{\overline{\al}}}, b_{t_{\overline{\beta}}}\in \C^{\times}$. Then, using (\ref{eq2.1})-(\ref{eqi2.5}), we have
\begin{align}
  \overline{h}&=-2\beta_{2}\left(\overline{h}\right)\overline{\al}\left(\overline{h}\right)-2\overline{\ptl}\left(\overline{\al}\left(\overline{h}\right)\right)
 \overline{\beta}\left(\overline{h}\right),  \label{eq2.7}\\
  \overline{h}&=-2\al_{2}\left(\overline{h}\right)\overline{\beta}\left(\overline{h}\right)-2\overline{\ptl}\left(\overline{\beta}\left(\overline{h}\right)\right)
 \overline{\al}\left(\overline{h}\right), \label{eq2.8} \\
   h&=-2\left(2\kappa_{1}\left(\overline{h}\right)+\kappa_{2}\left(\overline{h}\right)
   +\kappa_{3}\left(\overline{h}\right)\right)h-\eta\left(\overline{h}\right),\label{eq2.9}\\
 0&=\left(\al_{2}\left(\overline{h}\right)-\overline{\partial}\left(\overline{\al}\left(\overline{h}\right)\right)\right)\overline{\al}\left(\overline{h}\right),\label{eq22.9}\\
 0&=\left(\beta_{2}\left(\overline{h}\right)-\overline{\partial}\left(\overline{\beta}\left(\overline{h}\right)\right)\right)\overline{\beta}\left(\overline{h}\right), \label{eq22.10}
\end{align}
where $\kappa_{1}\left(\overline{h}\right):=\al_{2}\left(\overline{h}\right)\beta_{2}\left(\overline{h}\right)$, $\kappa_{2}\left(\overline{h}\right):= \overline{\partial}\left(\beta_{2}\left(\overline{h}\right)\right)\overline{\al}\left(\overline{h}\right)$, $\kappa_{3}\left(\overline{h}\right):= \overline{\ptl}\left(\al_{2}\left(\overline{h}\right)\right)\overline{\beta}\left(\overline{h}\right)$ and $\eta\left(\overline{h}\right):=2\al_{1}\left(\overline{h}\right)\beta_{2}\left(\overline{h}\right)
+2\al_{2}\left(\overline{h}\right)\beta_{1}\left(\overline{h}\right)
+2\overline{\ptl}\left(\beta_{1}\left(\overline{h}\right)\right)\overline{\al}\left(\overline{h}\right)-4\overline{\ptl}\left(\beta_{2}\left(\overline{h}\right)\right)\overline{\al}\left(\overline{h}\right)
+2\overline{\ptl}\left(\al_{1}\left(\overline{h}\right)+2\al_{2}\left(\overline{h}\right)\right)\overline{\beta}\left(\overline{h}\right)$.

Writing $\al_{2}\left(\overline{h}\right)=\sum_{p=0}^{n_{\al}}k_{p}\overline{h}^{p}$, $\beta_{2}\left(\overline{h}\right)=\sum_{q=0}^{n_{\beta}}k'_{q}\overline{h}^{q}$ and comparing the degree of $\overline{h}$ and the degree of $h$, respectively, in (\ref{eq2.7})-(\ref{eq22.10}), we obtain
\begin{equation}\label{eq2.10}
  -2k'_{n_{\beta}}a_{r_{\overline{\al}}}-2a_{r_{\overline{\al}}}r_{\overline{\al}}b_{t_{\overline{\beta}}}=
  \left\{
  \begin{array}{ll}
  0\ \ &\mbox{if}\ \ n_{\beta}+r_{\overline{\al}}=r_{\overline{\al}}+t_{\overline{\beta}}-1> 1,\vspace{1ex}\\
  1\ \ &\mbox{if}\ \ n_{\beta}+r_{\overline{\al}}=r_{\overline{\al}}+t_{\overline{\beta}}-1=1,
 \end{array}
\right.
\end{equation}
\begin{equation}\label{eq2.11}
  -2k_{n_{\al}}b_{t_{\overline{\beta}}}-2b_{t_{\overline{\beta}}}t_{\overline{\beta}}a_{r_{\overline{\al}}}=
  \left\{
  \begin{array}{ll}
  0\ &\mbox{if}\ \ n_{\al}+t_{\overline{\beta}}=r_{\overline{\al}}+t_{\overline{\beta}}-1> 1,\vspace{1ex}\\
  1\ &\mbox{if}\ \ n_{\al}+t_{\overline{\beta}}=r_{\overline{\al}}+t_{\overline{\beta}}-1=1,
 \end{array}
\right.
\end{equation}
\begin{equation}\label{eq2.12}
  -k'_{n_{\beta}}(2k_{n_{\al}}+n_{\beta}a_{r_{\overline{\al}}})-k_{n_{\al}}n_{\al}b_{t_{\overline{\beta}}}=
  \left\{
  \begin{array}{ll}
  0\ \ \ &\mbox{if}\ \ \,n_{\al}+n_{\beta}> 0,\vspace{1ex}\\
  \frac{1}{2}\ \ \ &\mbox{if}\ \ \,n_{\al}+n_{\beta}=0.
 \end{array}
\right.
\end{equation}
If $n_{\beta}+r_{\overline{\al}}> 1$, then $n_{\al}+t_{\overline{\beta}}> 1$ and $n_{\al}+n_{\beta}> 0$. Thus, due to (\ref{eq2.10}), (\ref{eq2.11}) and (\ref{eq2.12}), we obtain $r_{\overline{\al}}+t_{\overline{\beta}}=0$, which implies $r_{\overline{\al}}=t_{\overline{\beta}}=0$, a contradiction.

Therefore, it follows that $n_{\beta}+r_{\overline{\al}}=r_{\overline{\al}}+t_{\overline{\beta}}-1=n_{\al}+t_{\overline{\beta}}=1$ and $n_{\al}+n_{\beta}=0$, which yields $n_{\al}=n_{\beta}=0$ and $r_{\overline{\al}}=t_{\overline{\beta}}=1$. Thus, we can write $\overline{\al}\left(\overline{h}\right)=a_{1}\overline{h}+a_{0}$, $\overline{\beta}\left(\overline{h}\right)=b_{1}\overline{h}+b_{0}$, $\al\left(h,\overline{h}\right)=c_{1}h+\al_{1}\left(\overline{h}\right)$ and $\beta\left(h,\overline{h}\right)=d_{1}h+\beta_{1}\left(\overline{h}\right)$, for some $a_{1},b_{1},c_{1},d_{1}\in \C^{\times}$, $a_{0},b_{0}\in \C$ and $\al_{1}\left(\overline{h}\right),\beta_{1}\left(\overline{h}\right)\in \C\left[\,\overline{h}\,\right]$. According  to (\ref{eq2.7})-(\ref{eq22.10}), we have
\begin{align*}
  \overline{h} &= -2a_{1}(d_{1}+b_{1})\overline{h}-2(d_{1}a_{0}+a_{1}b_{0})= -2b_{1}(c_{1}+a_{1})\overline{h}-2(c_{1}b_{0}+b_{1}a_{0}), \\
  h &= -4d_{1}c_{1}h-2\left(d_{1}\al_{1}\left(\overline{h}\right)+c_{1}\beta_{1}\left(\overline{h}\right)
  +\overline{\partial}\left(\al_{1}\left(\overline{h}\right)\right)\overline{\beta}\left(\overline{h}\right)
  +\overline{\partial}(\beta_{1}\left(\overline{h}\right))\overline{\al}\left(\overline{h}\right)\right),\\
  0&=\left(c_{1}-a_{1}\right)\left(a_{1}\overline{h}+a_{0}\right)=\left(b_{1}-d_{1}\right)\left(b_{1}\overline{h}+b_{0}\right),
\end{align*}
which implies that $a_{1}=c_{1}=\frac{\lambda}{2}$, $b_{1}=d_{1}=-\frac{1}{2\lambda}$, $b_{0}=\frac{1}{\lambda^{2}}a_{0}$, $\frac{1}{2\lambda}\al_{1}\left(\overline{h}\right)-\frac{\lambda}{2}\beta_{1}\left(\overline{h}\right)
-\overline{\al}\left(\overline{h}\right)\overline{\partial}\left(\beta_{1}\left(\overline{h}\right)\right)-\overline{\beta}\left(\overline{h}\right)\overline{\partial}\left(\al_{1}\left(\overline{h}\right)\right)=0$.
Thus, we obtain \begin{align*}
             \al\left(h,\overline{h}\right)&=\frac{\lambda}{2}h+\al_{1}\left(\overline{h}\right),\ \ \ \ \ \ \ \ \ \,
             \overline{\al}\left(\overline{h}\right)=\frac{\lambda}{2}\left(\overline{h}+b\right),\\
              \beta\left(h,\overline{h}\right)&=-\frac{1}{2\lambda}h+\beta_{1}\left(\overline{h}\right),\ \ \ \ \ \ \overline{\beta}\left(\overline{h}\right)=-\frac{1}{2\lambda}\left(\overline{h}-b\right),
           \end{align*}
where $b:=\frac{2}{\lambda}a_{0}\in \C$ and $\al_{1}\left(\overline{h}\right)$, $\beta_{1}\left(\overline{h}\right)$ satisfy
\begin{equation}\label{e34}
\frac{1}{\lambda}\al_{1}\left(\overline{h}\right)-\lambda\beta_{1}\left(\overline{h}\right)
-\lambda\left(\overline{h}+b\right)\overline{\partial}\left(\beta_{1}\left(\overline{h}\right)\right)
+\frac{1}{\lambda}\left(\overline{h}-b\right)\overline{\partial}\left(\al_{1}\left(\overline{h}\right)\right)=0.
\end{equation}
From (\ref{e34}), we can write $\al_{1}\left(\overline{h}\right)=\sum_{i=0}^{m}p_{i}\overline{h}^{i}$ and $\beta_{1}\left(\overline{h}\right)=\sum_{j=0}^{m}q_{j}\overline{h}^{j}$ with $m\in \Z_{+}$ and $p_{i},q_{j}\in \C$ for $i,j\in \{0,\ldots, m\}$. Thus, according to (\ref{e34}), we obtain
\begin{equation}\label{e35}
\sum_{i=0}^{m}\frac{1}{\lambda}p_{i}(i+1)\overline{h}^{i}-\sum_{j=0}^{m}\lambda q_{j}(1+j)\overline{h}^{j}
-\sum_{i=0}^{m}\frac{b}{\lambda}p_{i}i\overline{h}^{i-1}-\sum_{j=0}^{m}\lambda bq_{j}j\overline{h}^{j-1}=0.
\end{equation}
Now we consider the following two cases.

\noindent{\bf Case 1:} $b\neq 0$. Comparing the coefficients of $\overline{h}^{i}$, for $i\in \{0,\ldots,m\}$, in (\ref{e35}), we have
\begin{align*}
       \frac{1}{\lambda}p_{0}-\lambda q_{0}-\lambda bq_{1}- \frac{b}{\lambda}p_{1}&= 0, \\
       \frac{1}{\lambda}p_{1}-\lambda q_{1}-\lambda bq_{2}- \frac{b}{\lambda}p_{2}&= 0, \\
       \cdots \cdots \\
        \frac{1}{\lambda}p_{m}-\lambda q_{m}&=0.
     \end{align*}
Thus, $\left(
\begin{array}{c}
p_{0} \\
p_{1} \\
p_{2} \\
\vdots \\
p_{m} \\
\end{array}
\right)=\lambda^{2}A\left(
                        \begin{array}{c}
                          q_{0} \\
                          q_{1} \\
                          q_{2} \\
                          \vdots \\
                          q_{m} \\
                        \end{array}
                      \right)$, where $A=\left(
                       \begin{array}{ccccc}
                         1 & 2b & 2b^{2} & \cdots & 2b^{m} \\
                         0 & 1 & 2b & \cdots & 2b^{m-1} \\
                         0 & 0 & 1 &\cdots  & 2b^{m-2} \\
                         \vdots & \vdots & \vdots& \ddots & \vdots \\
                         0 & 0 & 0 & \cdots & 1 \\
                       \end{array}
                     \right)$.

\noindent{\bf Case 2:} $b=0$. Due to (\ref{e35}), we get
\begin{equation}\label{e36}
\sum_{i=0}^{m}\frac{1}{\lambda}p_{i}(i+1)\overline{h}^{i}-\sum_{j=0}^{m}\lambda q_{j}(1+j)\overline{h}^{j}=0.
\end{equation}
Comparing the coefficients of $\overline{h}^{i}$, for $i\in \{0,\ldots,m\}$, in (\ref{e36}), we see that $p_{i}=\lambda^{2}q_{i}$ for all $i$, which implies that $\al_{1}\left(\overline{h}\right)=\lambda^{2}\beta_{1}\left(\overline{h}\right)$.

Thus, from the above two cases, the proof is completed.
\end{proof}

\subsubsection{Proof of Theorem \ref{prop-3.2a}} 
First, we prove Theorem \ref{prop-3.2a} (i). Here we prove the result for $\Gamma(\lambda,a,b)$. The proof for $\Theta(\lambda,a,b)$ is similar. Let $\Gamma_{1}$ be a nonzero submodule of $\Gamma(\lambda,a,b)$ and let $\gamma\left(h,\overline{h}\right)$ be a nonzero polynomial in $\Gamma_{1}$. From the action of $e$ on $\gamma\left(h,\overline{h}\right)$, there exists a nonzero polynomial $\gamma_{1}(h)\in \C[h]$ such that $\gamma_{1}(h)\in \Gamma_{1}$. Let $N(\gamma_{1}(h))$ be the set of zeros of $\gamma_{1}(h)$, this is a finite subset of $\C$. According to the definition of the module structure, we see that
\begin{equation*}
  N(\overline{e}\cdot\gamma_{1}(h))=N(\gamma_{1}(h))+2,
\end{equation*}
and, inductively, we have $N\left(\overline{e}^{k}\cdot\gamma_{1}(h)\right)=N(\gamma_{1}(h))+2k$. Now take $k$ large enough such that
\begin{equation*}
\left(N(\gamma_{1}(h))+2k\right)\cap N(\gamma_{1}(h))=\emptyset.
\end{equation*}
Then $\gamma_{1}(h)$ and $\overline{e}^{k}\cdot\gamma_{1}(h)$ are relatively prime elements of $\Gamma_{1}$. Thus we can find $\kappa_{1}(h)$, $\kappa_{2}(h)\in \C[h]$ such that
\begin{equation*}
  \kappa_{1}(h)\gamma_{1}(h)+\kappa_{2}(h)\overline{e}^{k}\cdot\gamma_{1}(h)=1\in \Gamma_{1}.
\end{equation*}
Therefore, we get $\Gamma_{1}=\C\left[\,h,\overline{h}\,\right]$, which implies that $\Gamma(\lambda,a,b)$ is a simple $\g$-module.

Next, we prove that $\Gamma(\lambda,a,b)\cong \Gamma(\lambda',a',b')$ if and only if $\lambda=\lambda'$, $a=a'$ and $b=b'$. The sufficiency is obvious. Suppose that $\Gamma(\lambda,a,b)\cong \Gamma(\lambda',a',b')$ as $\g$-modules and let $\phi:\Gamma(\lambda,a,b)\longrightarrow \Gamma(\lambda',a',b')$ be an isomorphism of $\g$-modules.

For any $i,j\in \Z_{+}$, from $\phi\left(h^{i}\overline{h}^{j}\cdot 1\right)=h^{i}\overline{h}^{j}\phi\left(1\right)$, we obtain that $\phi\left(\C\left[\,h,\overline{h}\,\right]\right)=\C\left[h,\overline{h}\,\right]\phi\left(1\right)$. Since $\phi$ is an isomorphism, it follows that $\phi(1)\in \C^{\times}$. Taking $\kappa\left(h,\overline{h}\right)\in \Gamma(\lambda,a,b)$, we have
\begin{align*}
 \phi\big(\overline{e}\cdot \kappa\left(h,\overline{h}\right)\big) &= \phi\left(\lambda \kappa\left(h-2,\overline{h}\right)\right)= \lambda\kappa\left(h-2,\overline{h}\right)\phi(1),\\
  \overline{e}\cdot \phi\left(\kappa\left(h,\overline{h}\right)\right) &=\lambda'\kappa\left(h-2,\overline{h}\right)\phi(1),\\
  \phi\left(\overline{f}\cdot \kappa\left(h,\overline{h}\right)\right) &=-\frac{1}{4\lambda}\phi\left( \left(\overline{h}^{2}+a\right)\kappa\left(h+2,\overline{h}\right)\right)=-\frac{1}{4\lambda} \left(\overline{h}^{2}+a\right)\kappa\left(h+2,\overline{h}\right)\phi(1),\\
  \overline{f}\cdot \phi\left(\kappa\left(h,\overline{h}\right)\right) &=-\frac{1}{4\lambda'}\left(\overline{h}^{2}+a'\right)\kappa\left(h+2,\overline{h}\right)\phi\left(1\right),
 \end{align*}
 which implies that $\lambda=\lambda'$ and $a=a'$. Moreover, for any $\kappa\left(h,\overline{h}\right)\in \Gamma(\lambda,a,b)$, we get
 \begin{align*}
  \phi\left(f\cdot \kappa\left(h,\overline{h}\right)\right) 
  &=-\frac{1}{2\lambda}\left(\left((h+2)\overline{h}+b\right)\kappa\left(h+2,\overline{h}\right)
  +\left(\overline{h}^{2}+a\right)\overline{\partial}\left(\kappa\left(h+2,\overline{h}\right)\right)\right)\phi\left( 1\right),\\
  f\cdot\phi\left(\kappa\left(h,\overline{h}\right)\right) 
  &=-\frac{1}{2\lambda}\left(\left((h+2)\overline{h}+b'\right)\kappa\left(h+2,\overline{h}\right)+\left(\overline{h}^{2}
  +a\right) \overline{\partial}\left(\kappa\left(h+2,\overline{h}\right)\right)\right)\phi\left( 1\right),
\end{align*}
which implies that $b=b'$. This completes the proof of Theorem \ref{prop-3.2a} (i). 

Now we give the proof of Theorem \ref{prop-3.2a} (ii). Suppose that $\Omega_{1}$ is a nonzero submodule of $\Omega\left(\lambda,b,\beta_{1}\left(\overline{h}\right)\right)$. Let $\gamma\left(h,\overline{h}\right)$ be a nonzero polynomial in $\Omega_{1}$ with the smallest degree of $h$ in $\Omega_{1}$. By the action $\overline{e}$ on $\gamma\left(h,\overline{h}\right)$, we have
\begin{equation*}
  \left(\frac{\lambda}{2}\left(\overline{h}+b\right)-\overline{e}\right)\cdot\gamma\left(h,\overline{h}\right)=\frac{\lambda}{2}\left(\overline{h}+b\right)
  \left(\gamma\left(h,\overline{h}\right)-\gamma\left(h-2,\overline{h}\right)\right)\in\Omega_{1}.
\end{equation*}
This implies $\gamma\left(h,\overline{h}\right)\in \C\left[\,\overline{h}\,\right]$ by the choice of $\gamma\left(h,\overline{h}\right)$. Denote $\gamma\left(\overline{h}\right)=\gamma\left(h,\overline{h}\right)$ and assume $\mbox{deg}_{\overline{h}}\gamma\left(\overline{h}\right)=m$. If $m=0$, we are done. Thus, we assume $m>0$. Now we consider the following two cases.

\noindent{\bf Case 1:} $b\neq0$. Applying $f$ on $\gamma\left(\overline{h}\right)$, we obtain
\begin{equation*}
  \left(\overline{h}-b\right)\overline{\partial}\left(\gamma\left(\overline{h}\right)\right)=-\lambda f\cdot\gamma\left(\overline{h}\right)-\lambda\left(\frac{1}{2\lambda}h-\beta_{1}\left(\overline{h}\right)\right)\gamma\left(\overline{h}\right)
  \in \Omega_{1}.
\end{equation*}
Next, applying $f$ on $\left(\overline{h}-b\right)\overline{\partial}\left(\gamma\left(\overline{h}\right)\right)$, we see that $\left(\overline{h}-b\right)^{2}\overline{\partial}^{2}\left(\gamma\left(\overline{h}\right)\right)\in \Omega_{1}$. Inductively, we have $\left(\overline{h}-b\right)^{m}\in \Omega_{1}$. Due to the action of $e$ on $\left(\overline{h}-b\right)^{m}$, it follows that
\begin{equation*}
  \left(\overline{h}+b\right)\left(\overline{h}-b\right)^{m-1}=-\frac{1}{m\lambda}e\cdot \left(\overline{h}-b\right)^{m}+\frac{1}{m\lambda}\left(\frac{\lambda}{2}h+\al_{1}\left(\overline{h}\right)\right)\left(\overline{h}-b\right)^{m}\in \Omega_{1}.
\end{equation*}
Thus, we get $\left(\overline{h}-b\right)^{m-1}=\frac{1}{2b}\left(\left(\overline{h}+b\right)\left(\overline{h}-b\right)^{m-1}-\left(\overline{h}-b\right)^{m}\right)\in \Omega_{1}$. Then, applying $e$ to $\left(\overline{h}-b\right)^{m-1}$, it follows that $\left(\overline{h}-b\right)^{m-2}\in \Omega_{1}$. Continuing this procedure finitely many times, we obtain $1\in \Omega_{1}$. Therefore, it follows that $\Omega_{1}=\Omega\left(\lambda,b,\beta_{1}\left(\overline{h}\right)\right)$, which means that $\Omega\left(\lambda,b,\beta_{1}\left(\overline{h}\right)\right)$ is a simple $\g$-module.

\noindent{\bf Case 2:} $b=0$. It is clear that, for any $i\in \Z_{+}$, the space $\overline{h}^{i}\C\left[\,h,\overline{h}\,\right]$ is a submodule of $\Omega\left(\lambda,0,\beta_{1}\left(\overline{h}\right)\right)$, for $\lambda\in \C^{\times}$ and $\beta_{1}\left(\overline{h}\right)\in \C\left[\,\overline{h}\,\right]$. Moreover, the quotient module $\Omega_{2}:=\overline{h}^{i}\C\left[\,h,\overline{h}\,\right]/\overline{h}^{i+1}\C\left[\,h,\overline{h}\,\right]$ is simple if and only if $2\left(-\lambda\beta_{1}\left(\overline{h}\right)+i\right)\notin \Z_{+}$, for any $i\in \Z_{+}$. Indeed, for any $0\neq \gamma\left(h,\overline{h}\right)\in \Omega_{2}$, we have
\begin{align*}
\overline{e}\cdot\gamma\left(h,\overline{h}\right)&\equiv0\ \ \ \mbox{mod}\ \  \overline{h}^{i+1}\C\left[\,h,\overline{h}\,\right],\\
 \overline{f}\cdot \gamma\left(h,\overline{h}\right)&\equiv0\ \ \ \mbox{mod}\ \  \overline{h}^{i+1}\C\left[\,h,\overline{h}\,\right],\\
e\cdot\gamma\left(h,\overline{h}\right)&=\lambda\left(\frac{h}{2}
+\frac{1}{\lambda}\al_{1}\left(\overline{h}\right)-i\right)\gamma\left(h-2,\overline{h}\right),\\
f\cdot\gamma\left(h,\overline{h}\right)&=-\frac{1}{\lambda}\left(\frac{h}{2}-\lambda\beta_{1}\left(\overline{h}\right)+i\right)\gamma\left(h+2,\overline{h}\right).
\end{align*}
Hence, $\Omega_{2}\cong \Delta_{1}\left(-\frac{1}{\lambda},-\lambda\beta_{1}\left(\overline{h}\right)+i\right)$ as $\mathfrak{sl}_{2}$-modules. Thus, from Theorem \ref{theo1.1}, we see that $\Omega_{2}$ is simple if and only if $2\left(-\lambda\beta_{1}\left(\overline{h}\right)+i\right)\notin \Z_{+}$.

Next, we prove that $\Omega\left(\lambda,b,\beta_{1}\left(\overline{h}\right)\right)\cong \Omega(\lambda',b',\beta'_{1}\left(\overline{h})\right)$ if and only if $\lambda=\lambda'$, $b=b'$ and $\beta_{1}\left(\overline{h}\right)=\beta'_{1}\left(\overline{h}\right)$, where $\lambda,\lambda'\in \C^{\times}$, $b,b'\in \C$ and $\beta_{1}\left(\overline{h}\right),\beta'_{1}\left(\overline{h}\right)\in \C\left[\,\overline{h}\,\right]$. The sufficiency is clear. Suppose that $\varphi:\Omega\left(\lambda,b,\beta_{1}\left(\overline{h}\right)\right)\longrightarrow \Omega\left(\lambda',b',\beta'_{1}\left(\overline{h}\right)\right)$ is an isomorphism of $\g$-modules. Let $\gamma\left(h,\overline{h}\right)\in \C\left[\,h,\overline{h}\,\right]$. Then, we have
\begin{align*}
  \varphi\left(\overline{e}\cdot \gamma\left(h,\overline{h}\right)\right) &=\varphi\left(\frac{\lambda}{2}\left(\overline{h}+b\right)\gamma\left(h-2,\overline{h}\right)\right)
  =\frac{\lambda}{2}\left(\overline{h}+b\right)\gamma\left(h-2,\overline{h}\right)\varphi(1), \\
 \overline{e}\cdot\varphi\left(\gamma\left(h,\overline{h}\right)\right) &=\frac{\lambda'}{2}\left(\overline{h}+b'\right)\varphi\left(\gamma\left(h-2,\overline{h}\right)\right)
 =\frac{\lambda'}{2}\left(\overline{h}+b'\right)\gamma\left(h-2,\overline{h}\right)\varphi\left(1\right),
\end{align*}
which implies that $\lambda=\lambda'$ and $b=b'$. Moreover, due to the action of $f$ on $\C\left[\,h,\overline{h}\,\right]$, we get
\begin{align*}
  \varphi\left(f\cdot \gamma\left(h,\overline{h}\right)\right) &= -\varphi\left(\left(\frac{h}{2\lambda}-\beta_{1}\left(\overline{h}\right)\right)\gamma\left(h+2,\overline{h}\right)
  +\frac{1}{\lambda}\left(\overline{h}-b\right)\overline{\partial}\left(\gamma\left(h+2,\overline{h}\right)\right)\right)\\
  &=-\left(\frac{h}{2\lambda}-\beta_{1}\left(\overline{h}\right)\right)\gamma\left(h+2,\overline{h}\right)\varphi(1)
  -\frac{1}{\lambda}\left(\overline{h}-b\right)\overline{\partial}\left(\gamma\left(h+2,\overline{h}\right)\right)\varphi(1), \\
  f\cdot\varphi\left(\gamma\left(h,\overline{h}\right)\right)&=-\left(\frac{h}{2\lambda'}
  -\beta'_{1}\left(\overline{h}\right)\right)\gamma\left(h+2,\overline{h}\right)\varphi(1)
   -\frac{1}{\lambda'}\left(\overline{h}-b'\right)\overline{\partial}\left(\gamma\left(h+2,\overline{h}\right)\right)\varphi(1),
\end{align*}
yielding that $\beta_{1}\left(\overline{h}\right)=\beta'_{1}\left(\overline{h}\right)$, as desired. This completes the proof.\hfill\qed
\section{Three families of weight modules over $\g$}
In this section, we investigate three families of weight modules over $\g$ associated to the non-weight modules $\Gamma(\lambda,a,b)$,  $\Theta(\lambda,a,b)$ and $\Omega\left(\lambda,b,\beta_{1}\left(\overline{h}\right)\right)$, respectively, which are defined in Definition $\ref{defi01}$. We also compare these modules to each other using  Mathieu's twisting functors from \cite{M}.
\subsection{Weight modules associated to $\Gamma(\lambda,a,b)$}
 In the following, we denote $\delta_{l,r}$ the Kronecker delta for any $l,r\in \Z_{+}$. Let $\al,\b\in \C$, $k\in \Z$, $s\in \N$ and $\eta_{\al+2k,\b^{s}}:\C\left[\,h,\overline{h}\,\right]\longrightarrow \C$ be the linear transformation defined by
\begin{equation*}
  \eta_{\al+2k,\b^{s}}\left((h-\al-2k)^{i}\left(\overline{h}-\b\right)^{j}\right)=\delta_{i,0}\delta_{j,s-1}(s-1)!.
\end{equation*}
In particular, $\eta_{\al+2k,\b}$ (in the case $s=1$) is an algebra homomorphism, for any $k\in \Z$. For convenience, in what follows, we denote $\al+2k$ by $\al_{k}$, for $k\in \Z$.

For $\al,\b, a, b \in \C$ and $\lambda\in \C^{\times}$, let $M_{\al,\b}^{\lambda,a,b}$ be the submodule of $\Gamma(\lambda,a,b)^{*}$ spanned by a linear independent collection $\{\eta_{\al_{k},\b^{s}}\,|\,k\in \Z, s\in \N\}$. Then $\{\eta_{\al_{k},\b^{s}}\,|\,k\in \Z, s\in \N\}$ is, in fact, a basis of $M_{\al,\b}^{\lambda,a,b}$. Then $M_{\al,\b}^{\lambda,a,b}=\bigoplus_{k\in \Z}M_{\al,\b}^{k}$, where $M_{\al,\b}^{k}:=\mbox{span}\{\eta_{\al_{k},\b^{s}}\,|\,s\in \N\}$, for $k\in \Z$. The following proposition gives a precise action of $\g$ on $M_{\al,\b}^{\lambda,a,b}$.
\begin{prop}\label{defi3.1}Let $\lambda\in \C^{\times}$ and $\al,\b,a,b\in \C$. Then, for any $k\in \Z$ and $s\in \N$, the space
$M_{\al,\b}^{\lambda,a,b}$ is a $\g$-module with the action defined as follows:
\begin{eqnarray}&\!\!\!\!\!\!\!\!\!\!\!\!\!\!\!\!\!\!\!\!\!\!\!\!&
  \overline{e}\cdot \eta_{\al_{k},\b^{s}}= -\lambda\eta_{\al_{k-1},\b^{s}},\ \ \ \ \ \ \ \ \ \ \ \ \ \ \ e\cdot \eta_{\al_{k},\b^{s}}= 2\lambda\eta_{\al_{k-1},\b^{s+1}},\label{vz4.11}\\
 &\!\!\!\!\!\!\!\!\!\!\!\!\!\!\!\!\!\!\!\!\!\!\!\!&
 h\cdot \eta_{\al_{k},\b^{s}} =-\al_{k}\eta_{\al_{k},\b^{s}}, \ \ \ \ \ \ \ \ \ \ \ \ \ \ \ \overline{h}\cdot \eta_{\al_{k},\b^{s}} =-\b\eta_{\al_{k},\b^{s}}-(s-1)\eta_{\al_{k},\b^{s-1}},\label{vz14.4}\\
 &\!\!\!\!\!\!\!\!\!\!\!\!\!\!\!\!\!\!\!\!\!\!\!\!&
 \overline{f}\cdot \eta_{\al_{k},\b^{s}} = \frac{(s-2)(s-1)}{4\lambda}\eta_{\al_{k+1},\b^{s-2}}+\frac{(s-1)\b}{2\lambda}\eta_{\al_{k+1},\b^{s-1}}
 +\frac{\b^{2}+a}{4\lambda}\eta_{\al_{k+1},\b^{s}},\label{vz14.5}\\
 &\!\!\!\!\!\!\!\!\!\!\!\!\!\!\!\!\!\!\!\!\!\!\!\!&
 f\cdot\eta_{\al_{k},\b^{s}}= \frac{(s-1)(\al_{k}+s)}{2\lambda}\eta_{\al_{k+1},\b^{s-1}}+\frac{\al_{k+s}\b+b}{2\lambda}\eta_{\al_{k+1},\b^{s}}
  +\frac{\b^{2}+a}{2\lambda}\eta_{\al_{k+1},\b^{s+1}}.\label{vz14.6}
\end{eqnarray}\end{prop}
\begin{proof}Let $i,j\in \Z_{+}$, $k\in \Z$, $s\in \N$ and $\gamma\left(h,\overline{h}\right)=:\left(h-\al_{k}\right)^{i}\left(\overline{h}-\b\right)^{j}$. Then, according to Definition $\ref{defi01}$, we obtain that
\begin{align*}
    \overline{e} \cdot\eta_{\al_{k},\b^{s}}\left(\gamma\left(h+2,\overline{h}\right)\right)
       &= -\lambda \eta_{\al_{k},\b^{s}}\left(\left(h-\al_{k}\right)^{i}\left(\overline{h}-\b\right)^{j}\right)=-\lambda\delta_{i,0}\delta_{j,s-1} (s-1)!,\\
   e\cdot \eta_{\al_{k},\b^{s}}\left(\gamma\left(h+2,\overline{h}\right)\right)
       &=2\lambda j\eta_{\al_{k},\b^{s}}\left(\left(h-\al_{k}\right)^{i}\left(\overline{h}-\b\right)^{j-1}\right)  =2\lambda\delta_{i,0}\delta_{j,s}s!.
\end{align*}
Thus we get that $\overline{e}\cdot \eta_{\al_{k},\b^{s}}= -\lambda\eta_{\al_{k-1},\b^{s}}$ and $e\cdot \eta_{\al_{k},\b^{s}}= 2\lambda\eta_{\al_{k-1},\b^{s+1}}$. Next, from the action of $h,\overline{h}$ on $\Gamma(\lambda,a,b)$, respectively, we have
\begin{align*}
(h+\al_{k}) \cdot\eta_{\al_{k},\b^{s}}\left(\gamma\left(h,\overline{h}\right)\right)
       &= -\eta_{\al_{k},\b^{s}}\left(\left(h-\al_{k}\right)^{i+1}\left(\overline{h}-\b\right)^{j}\right)=0,\\
\left(\overline{h}+\b\right)\cdot \eta_{\al_{k},\b^{s}}\left(\gamma\left(h,\overline{h}\right)\right)
       &=-\eta_{\al_{k},\b^{s}}\left(\left(h-\al_{k}\right)^{i}\left(\overline{h}-\b\right)^{j+1}\right)=-\delta_{i,0}\delta_{j,s-2}(s-1)!.
\end{align*}
Thus it follows that $h\cdot \eta_{\al_{k},\b^{s}}= -\al_{k}\eta_{\al_{k},\b^{s}}$ and $\overline{h}\cdot \eta_{\al_{k},\b^{s}}= -\b \eta_{\al_{k},\b^{s}}-(s-1)\eta_{\al_{k},\b^{s-1}}$. Due to the action of $\overline{f}$ on $\Gamma(\lambda,a,b)$, we get the following \begin{eqnarray*}\!\!\!\!\!\!\!\!\!\!\!\!&\!\!\!\!\!\!\!\!\!\!\!\!\!\!\!&
 \ \ \ \ \overline{f} \cdot\eta_{\al_{k},\b^{s}}\left(\gamma\left(h-2,\overline{h}\right)\right)
  =\frac{1}{4\lambda}\eta_{\al_{k},\b^{s}}\left( \left(\overline{h}^{2}+a\right)\left(h-\al_{k}\right)^{i}\left(\overline{h}-\b\right)^{j}\right)
\nonumber\\\!\!\!\!\!\!\!\!\!\!\!\!&\!\!\!\!\!\!\!\!\!\!\!\!\!\!\!&
=\frac{1}{4\lambda}\eta_{\al_{k},\b^{s}}\left(\left(h-\al_{k}\right)^{i}\left(\left(\overline{h}-\b\right)^{j+2}
      +2\b\left(\overline{h}-\b\right)^{j+1}+\left(\b^{2}+a\right)\left(\overline{h}-\b\right)^{j}\right)\right)
\nonumber\\\!\!\!\!\!\!\!\!\!\!\!\!&\!\!\!\!\!\!\!\!\!\!\!\!\!\!\!&
  =\frac{1}{4\lambda}\delta_{i,0}\left(\delta_{j,s-3}+2\b\delta_{j,s-2}+\left(\b^{2}+a\right)\delta_{j,s-1}\right)(s-1)!.
\end{eqnarray*}
Furthermore, according to Definition $\ref{defi01}$, it follows that
\begin{eqnarray*}\!\!\!\!\!\!\!\!\!\!\!\!&\!\!\!\!\!\!\!\!\!\!\!\!\!\!\!&
\ \ \ \ \ \ \ \ \ \ \ \ f \cdot\eta_{\al_{k},\b^{s}}\left(\gamma\left(h-2,\overline{h}\right)\right)
 =\frac{1}{2\lambda}\eta_{\al_{k},\b^{s}}\left(\left((h+2)\overline{h}+b\right)\gamma\left(h,\overline{h}\right)
  +\left(\overline{h}^{2}+a\right)\partial\left(\gamma\left(h,\overline{h}\right)\right)\right) 
\nonumber\\\!\!\!\!\!\!\!\!\!\!\!\!&\!\!\!\!\!\!\!\!\!\!\!\!\!\!\!&
 \ \ \ \ \ \ \ \ \ =\frac{1}{2\lambda}\eta_{\al_{k},\b^{s}}\left(\left(h-\al_{k}\right)^{i}\left((\al_{k+1}+j)\left(\overline{h}-\b\right)^{j+1}
+((\al_{k+1}+2j)\b+b)\left(\overline{h}-\b\right)^{j}+j\left(\b^{2}+a\right)\left(\overline{h}-\b\right)^{j-1}\right)\right)
\nonumber\\\!\!\!\!\!\!\!\!\!\!\!\!&\!\!\!\!\!\!\!\!\!\!\!\!\!\!\!&
 \ \ \ \ \ \ \ \ \ =\frac{1}{2\lambda}\left((\al_{k}+s)\delta_{j,s-2}+(\al_{k+s}\b+b)\delta_{j,s-1}\right)\delta_{i,0}(s-1)!
 + \frac{\b^{2}+a}{2\lambda}\delta_{i,0}\delta_{j,s}s!.\end{eqnarray*}
 Thus, we can conclude that equations (\ref{vz14.5}), (\ref{vz14.6}) hold, completing the proof.
\end{proof}
Based on this proposition, we can give the proof of Theorem \ref{vztheo4.1}, which gives the sufficient and
necessary conditions for the $\g$-module $M_{\al,\b}^{\lambda,a,b}$ to be simple, and also describes its submodule structure if it is not simple. Additionally, the isomorphism class of $M_{\al,\b}^{\lambda,a,b}$ is determined in the theorem.
\subsubsection{Proof of Theorem \ref{vztheo4.1}}
First, we give the proof of statement (i) in Theorem \ref{vztheo4.1}. 
Let $N$ be a nonzero submodule of $M^{\lambda,a,b}_{\al,\b}$ and $0\neq n\in N$. Then we can write
\begin{equation*}
 n=\sum_{k\in \Z,s\in \N}c_{k,s}\eta_{\al_{k},\b^{s}}\ \ \ \mbox{for\ some\ }c_{k,s}\in \C.
\end{equation*}
Applying $h$ to $n$, 
we get $n_{1}:=\sum_{s\in \N}d_{k,s}\eta_{\al_{k},\b^{s}}\in M_{\al,\b}^{k}\cap N$ for some $k\in \Z$. Then, according to (\ref{vz14.4}), it follows that $\eta_{\al_{k},\b}\in M_{\al,\b}^{k}\cap N$. Next, from Proposition \ref{defi3.1}, we obtain $N=M^{\lambda,a,b}_{\al,\b}$ if and only if $\b^{2}+a\neq0$ or $\b^{2}+a=0$ and $(\al_{k}+2s)\b+b\neq0$ for any $k\in \Z$ and $s\in \N$.

Next, we show that statement (ii) holds. Let $M_{\al,\b}^{\lambda,a,b}$ be a $\g$-module, which is not simple. Then we obtain that $\b^{2}+a=0$ and $(\al_{k}+2s)\b+b=0$ for some $k\in \Z$ and $s\in \N$. 
Now we consider the following two cases.

\noindent{\bf Case 1:} $\b=a=0$. This implies that $b=0$. Thus from (\ref{vz14.4})-(\ref{vz14.6}), we get  $\overline{f}\cdot \eta_{\al_{k},\b}=0$, $f\cdot \eta_{\al_{k},\b}=0$ and $h\cdot \eta_{\al_{k},\b}=-\al_{k}\eta_{\al_{k},\b}$. Thus $M_{k,s}=U(\g)\eta_{\al_{k},\b}$ is the (opposite) Verma module with the lowest weight $\eta_{\al_{k},\b}$.

\noindent{\bf Case 2:} $\b\neq0$ and $a\neq0$. If $s=1$, by (\ref{vz14.4})-(\ref{vz14.6}), we obtain $\overline{f}\cdot \eta_{\al_{k},\b}=0$, $f\cdot \eta_{\al_{k},\b}=0$ and $h\cdot \eta_{\al_{k},\b}=-\al_{k}\eta_{\al_{k},\b}$. Therefore, $M_{k,s}$ is the Verma module with
the lowest weight $\eta_{\al_{k},\b}$.

If $s>1$, then take $k_{0}=k+s-1$ and we get
\begin{equation*}
  (\al+2k_{0}+2)\b+b=(\al+2k+2s)\b+b=0.
\end{equation*}
Thus, according to (\ref{vz14.4})-(\ref{vz14.6}), we get  $\overline{f}\cdot \eta_{\al+2k_{0},\b}=0$, $f\cdot \eta_{\al+2k_{0},\b}=0$ and $h\cdot \eta_{\al+2k_{0},\b}=-(\al+2k_{0})\eta_{\al+2k_{0},\b}$. Thus $U(\g)\eta_{\al+2k_{0},\b}$ is the Verma module with the lowest weight $\eta_{\al+2k_{0},\b}$. Combining
Case 1 and Case 2 implies statement (ii).

Next we prove statement (iii). Let $\varphi:M_{\al,\b}^{\lambda,a,b}\longrightarrow M_{\al,\b}^{\lambda',a,b}$ be a linear map defined by
\begin{equation*}
 \varphi:\eta_{\al_{k},\b^{s}} \longmapsto \frac{\lambda^{k}}{(\lambda')^{k}}\eta_{\al_{k},\b^{s}}.
\end{equation*}
The map  $\varphi$ is well defined and is an isomorphism of vector spaces. Now we show that it is a homomorphism of $U(\g)$-modules. For any $k\in \Z$ and $s\in \N$, we have the following
\begin{align*}
  \overline{e}\cdot \varphi\left(\eta_{\al_{k},\b^{s}}\right) &= \frac{\lambda^{k}}{(\lambda')^{k}}\overline{e}\cdot \eta_{\al_{k},\b^{s}}=-\frac{\lambda^{k}}{(\lambda')^{k-1}}\eta_{\al_{k-1},\b^{s}}, \\
  \varphi\left(\overline{e}\cdot\eta_{\al_{k},\b^{s}}\right)&=-\lambda\varphi\left(\eta_{\al_{k-1},\b^{s}}\right)
                   =-\frac{\lambda^{k}}{(\lambda')^{k-1}}\eta_{\al_{k-1},\b^{s}}, \\
  e\cdot \varphi\left(\eta_{\al_{k},\b^{s}}\right) &= \frac{\lambda^{k}}{(\lambda')^{k}}e\cdot \eta_{\al_{k},\b^{s}}=\frac{2\lambda^{k}}{(\lambda')^{k-1}}\eta_{\al_{k-1},\b^{s+1}},  \\
  \varphi\left(e\cdot\eta_{\al_{k},\b^{s}}\right)&=2\lambda\varphi\left(\eta_{\al_{k-1},\b^{s+1}}\right)
                   =\frac{2\lambda^{k}}{(\lambda')^{k-1}}\eta_{\al_{k-1},\b^{s+1}},   \\
  h\cdot \varphi\left(\eta_{\al_{k},\b^{s}}\right) &= \frac{\lambda^{k}}{(\lambda')^{k}}h\cdot \eta_{\al_{k},\b^{s}}=-\frac{\al_{k}\lambda^{k}}{(\lambda')^{k}}\eta_{\al_{k},\b^{s}},  \\
   \varphi\left(h\cdot\eta_{\al_{k},\b^{s}}\right)&=-\al_{k}\varphi\left(\eta_{\al_{k},\b^{s}}\right)
                   =-\frac{\al_{k}\lambda^{k}}{(\lambda')^{k}}\eta_{\al_{k},\b^{s}},  \\
  \overline{h}\cdot \varphi\left(\eta_{\al_{k},\b^{s}}\right) &= \frac{\lambda^{k}}{(\lambda')^{k}}\overline{h}\cdot \eta_{\al_{k},\b^{s}}=-\frac{\lambda^{k}}{(\lambda')^{k}}\left(\b\eta_{\al_{k},\b^{s}}+(s-1)\eta_{\al_{k},\b^{s-1}}\right),   \\
  \varphi\left(\overline{h}\cdot\eta_{\al_{k},\b^{s}}\right)
  &=-\frac{\lambda^{k}}{(\lambda')^{k}}\left(\b\eta_{\al_{k},\b^{s}}+(s-1)\eta_{\al_{k},\b^{s-1}}\right),
\end{align*}
which shows that $\varphi\left(y\cdot\eta_{\al_{k},\b^{s}}\right)=y\cdot\varphi\left(\eta_{\al_{k},\b^{s}}\right)$ for $y\in \left\{\overline{e},e,\overline{h},h\right\}$. Moreover, applying $\overline{f}$ on $\varphi\left(\eta_{\al_{k},\b^{s}}\right)$, we get
\begin{align*}
 \overline{f}\cdot \varphi\left(\eta_{\al_{k},\b^{s}}\right)&=\frac{\lambda^{k}}{4(\lambda')^{k+1}}\left((s-1)(s-2)\eta_{\al_{k+1},\b^{s-2}}+2(s-1)\b\eta_{\al_{k+1},\b^{s-1}}
+\left(\b^{2}+a\right)\eta_{\al_{k+1},\b^{s}}\right),\\
\varphi\left(\overline{f}\cdot\eta_{\al_{k},\b^{s}}\right)&=\frac{(s-1)(s-2)}{4\lambda}\varphi\left(\eta_{\al_{k+1},\b^{s-2}}\right)
 +\frac{(s-1)\b}{2\lambda}\varphi\left(\eta_{\al_{k+1},\b^{s-1}}\right)
  +\frac{\b^{2}+a}{4\lambda}\varphi\left(\eta_{\al_{k+1},\b^{s}}\right)\\
 &=\frac{\lambda^{k}}{4(\lambda')^{k+1}}\left((s-1)(s-2)\eta_{\al_{k+1},\b^{s-2}}+2(s-1)\b\eta_{\al_{k+1},\b^{s-1}}
+\left(\b^{2}+a\right)\eta_{\al_{k+1},\b^{s}}\right),
\end{align*}
 which implies that $\varphi\left(\overline{f}\cdot\eta_{\al_{k},\b^{s}}\right)=\overline{f}\cdot\varphi\left(\eta_{\al_{k},\b^{s}}\right)$. Moreover, let $f$ act on $\varphi\left(\eta_{\al_{k},\b^{s}}\right)$, we obtain
\begin{align*}
  f\cdot \varphi\left(\eta_{\al_{k},\b^{s}}\right)
  &=\frac{\lambda^{k}}{2(\lambda')^{k+1}}\left((s-1)(\al_{k}+s)\eta_{\al_{k+1},\b^{s-1}}+\left(\b^{2}+a\right)\eta_{\al_{k+1},\b^{s+1}}
  +\left(\al_{k+s}\b+b\right)\eta_{\al_{k+1},\b^{s}}\right),\\
\varphi\left(f\cdot\eta_{\al_{k},\b^{s}}\right)&=\frac{(s-1)(\al_{k}+s)}{2\lambda}\varphi\left(\eta_{\al_{k+1},\b^{s-1}}\right)
+\frac{\left(\b^{2}+a\right)}{2\lambda}\varphi\left(\eta_{\al_{k+1},\b^{s+1}}\right)
+\frac{\al_{k+s}\b+b}{2\lambda}\varphi\left(\eta_{\al_{k+1},\b^{s}}\right)\\
  &=\frac{\lambda^{k}}{2(\lambda')^{k+1}}\left((s-1)(\al_{k}+s)\eta_{\al_{k+1},\b^{s-1}}+\left(\b^{2}+a\right)\eta_{\al_{k+1},\b^{s+1}}
  +\left(\al_{k+s}\b+b\right)\eta_{\al_{k+1},\b^{s}}\right).
\end{align*}
Thus it follows that $\varphi\left(f\cdot\eta_{\al_{k},\b^{s}}\right)=f\cdot\varphi\left(\eta_{\al_{k},\b^{s}}\right)$. Therefore, $\varphi$ is a homomorphism of $U(\g)$-modules and the proof of statement (iii) is finished.

Now we prove statement (iv). By statement (iii), it is sufficient to show that $M_{\al,\b}^{\lambda,a,b}\cong M_{\al',\b'}^{\lambda,a',b'}$ if and only if $\al-\al'\in 2\Z$, $\b=\b'$, $a=a'$ and $b=b'$. The sufficiency of the conditions is clear. Suppose $M_{\al,\b}^{\lambda,a,b}\cong M_{\al',\b'}^{\lambda,a',b'}$ as $U(\g)$-modules and let $\psi:M_{\al,\b}^{\lambda,a,b}\longrightarrow M_{\al',\b'}^{\lambda,a',b'}$ be the isomorphism. Comparing the set of weights, we have $\al+2\Z=\al'+2\Z$, that is
\begin{equation}\label{vz4.71}
  \al-\al'\in 2\Z.\end{equation}
Assume $\al=\al'+2n_{0}$ for some $n_{0}\in \Z$. For any $k\in \Z$, applying $h$ on $\psi\left(\eta_{\al_{k},\b} \right)$, we get
 \begin{equation*}
   h \cdot\psi\left(\eta_{\al_{k},\b}\right)=\psi\left(h \cdot \eta_{\al_{k},\b}\right)=-\al_{k}\psi\left(\eta_{\al_{k},\b}\right)=-\left(\al'+2n_{0}+2k\right)\psi\left(\eta_{\al_{k},\b}\right).
 \end{equation*}
 Then it follows that $\psi\left(\eta_{\al_{k},\b}\right)\in M_{\al',\b'}^{n_{0}+k}$. Writing $\psi\left(\eta_{\al_{k},\b}\right)=\sum_{s\in \N}d_{k,s}\eta_{\al_{k},(\b')^{s}}$ with $d_{k,s}\in \C$ for $s\in \N$ and applying $\overline{h}$ to it, we obtain
 \begin{align*}
    \psi\left(\overline{h}\cdot \eta_{\al_{k},\b} \right) = -\b\psi\left(\eta_{\al_{k},\b} \right), \ \ \ \ \
   \overline{h}\cdot\psi\left( \eta_{\al_{k},\b} \right) 
   = -\b'\psi\left(\eta_{\al_{k},\b} \right)-\SUM{s\in \N}{}(s-1)d_{k,s}\eta_{\al_{k},(\b')^{s-1}}.
 \end{align*}
Comparing the respective coefficients of $\psi\left(\overline{h}\cdot \eta_{\al,\b} \right)$ and $ \overline{h}\cdot\psi\left( \eta_{\al,\b} \right)$, we obtain $s=1$ and
\begin{equation}\label{vz4.81}
  \b=\b'.
\end{equation}
This implies that $\psi\left(\eta_{\al_{k},\b} \right)=d_{k,1}\eta_{\al_{k},\b}$ with $d_{k,1}\in \C^{\times}$ for any $k\in \Z$. Acting by $\overline{e}$ on $\psi\left(\eta_{\al_{k},\b}\right)$, we have
\begin{align*}
  \psi\left(\overline{e}\cdot \eta_{\al_{k},\b}\right) =-\lambda\psi\left(\eta_{\al_{k-1},\b} \right)=-\lambda d_{k-1,1}\eta_{\al_{k-1},\b},\ \ \ \
\overline{e}\cdot \psi\left(\eta_{\al_{k},\b}\right) = d_{k,1}\overline{e}\cdot\eta_{\al_{k},\b}=-\lambda d_{k,1}\eta_{\al_{k-1},\b},
\end{align*}
which implies that $d_{k-1,1}=d_{k,1}$, for any $k\in \Z$.  Without loss of generality, we may assume $d_{k,1}=1$, for all $k\in \Z$. 
Next, applying $\overline{f}$ to $\psi\left(\eta_{\al-2,\b}\right)$, we have
\begin{align*}
 \psi\left(\overline{f}\cdot \eta_{\al-2,\b}\right) = \frac{\b^{2}+a}{4\lambda}\psi\left(\eta_{\al,\b} \right)=\frac{\b^{2}+a}{4\lambda}\eta_{\al,\b},\ \ \ \ \
  \overline{f}\cdot\psi\left(\eta_{\al-2,\b}\right) =\overline{f}\cdot\eta_{\al-2,\b}=\frac{\b^{2}+a'}{4\lambda}\eta_{\al,\b},
\end{align*}
which shows that
\begin{equation}\label{vzeq4.9}
  a=a'.
\end{equation}
Furthermore, acting by $e$ on $\psi\left(\eta_{\al+2,\b}\right)$ and by $f$ on $\psi\left(\eta_{\al-2,\b}\right)$, respectively, we obtain
\begin{align*}
   \psi\left(e\cdot \eta_{\al+2,\b}\right) &=2\lambda\psi\left(\eta_{\al,\b^{2}}\right), \ \ \ \ \ \ \ \ \ \ \ \ \ \ \ \ \ \ \ \ \ \ \ \ \ \ \ \ \,
  e\cdot\psi\left(\eta_{\al+2,\b}\right)=e\cdot\eta_{\al+2,\b}=2\lambda\eta_{\al,\b^{2}},\\
  \psi\left(f\cdot \eta_{\al-2,\b}\right)  
   &= \frac{\al\b+b}{2\lambda}\eta_{\al,\b}+\frac{\b^{2}+a}{2\lambda}\psi\left(\eta_{\al,\b^{2}}\right), \ \ \ \
  f\cdot\psi\left(\eta_{\al-2,\b}\right)
  =\frac{\al\b+b'}{2\lambda}
 \eta_{\al,\b}+\frac{\b^{2}+a}{2\lambda}\eta_{\al,\b^{2}}.
\end{align*}
 This implies $\psi\left(\eta_{\al,\b^{2}}\right)=\eta_{\al,\b^{2}}$  and $b=b'$. Thus, with (\ref{vz4.71}), (\ref{vz4.81}) and (\ref{vzeq4.9}), we complete the proof of statement (iv) and then the proof of Theorem \ref{vztheo4.1}.
\subsubsection{Proof of Theorem \ref{th0}}
Before proving Theorem \ref{th0}, we further need the following preliminary results for
later use. First, we consider the associative algebra $U^{\left(\overline{e}\right)}$, defined as the quotient of the free associative algebra $R$ with generators $\overline{e}^{-1}$, $\overline{e}$, $\overline{f}$, $\overline{h}$, $e$, $f$ and $h$ modulo the ideal, generated by the relations
\begin{align}
ef-fe&=h,\ \ \ \ \ \ he-eh=2e,\ \ \ \ \ \ e\overline{f}-\overline{f}e=\overline{h}, \ \ \ \ \ \ \ \ \ \ \overline{h}e-e\overline{h}=2\overline{e}, \label{vz4.1}\\
\overline{e}f-f\overline{e}&=\overline{h},\ \ \ \ \ fh-hf=2f, \ \ \ \ \ \ f\overline{h}-\overline{h}f=2\overline{f},\ \ \ \ \ \ \ \, h\overline{e}-\overline{e}h=2\overline{e}, \label{vz4.2}\\
\overline{f}h-h\overline{f}&=2\overline{f},\ \ \ \,y_{1}\overline{y_{1}}-\overline{y_{1}}y_{1}=0,\ \ \ \
\overline{y_{1}}\,\overline{y_{2}}-\overline{y_{2}}\,\overline{y_{1}}=0,\ \ \ \ \,\overline{e}^{-1}\,\overline{e}=\overline{e}\,\overline{e}^{-1}=1.\label{vz4.3}
\end{align}
where $y_{1},y_{2}\in \{e,h,f\}$. The algebra $U^{\left(\overline{e}\right)}$ is called the {\it localization} of the algebra $U(\g)$ with respect to the multiplicative set $\{\overline{e}^{i}\,:\,i\in \N\}$. As usual, abusing notation we will use the same notation for the elements of the original algebra $R$ and the quotient $U^{\left(\overline{e}\right)}$.

For $k\in \Z$, denote by $\Theta_{k}$ the automorphism of $U^{\left(\overline{e}\right)}$ given by the assignment $\Theta_{k}(u)=\overline{e}^{-k}u\overline{e}^{k}$,  where $ u\in U^{\left(\overline{e}\right)}$. Then we have the following Lemma.

\begin{lemm}\label{vzlemma4.4}For every $k\in \Z$, we have
\begin{align}
\Theta_{k}(\overline{e})&= \overline{e},\ \ \ \ \,\Theta_{k}\left(\overline{e}^{-1}\right)=\overline{e}^{-1},\ \ \ \ \, \Theta_{k}\left(\overline{f}\right) = \overline{f},\ \ \ \ \ \Theta_{k}\left(\overline{h}\right) = \overline{h},\label{vz4.4}\\
\Theta_{k}(e)&= e,\ \ \ \ \,\Theta_{k}(h) = h+2k,\ \ \ \ \Theta_{k}(f)= f-k\overline{e}^{-1}\overline{h}.\label{vz4.5}
\end{align}\end{lemm}
\begin{proof}According to the definition of $\Theta_{k}$ and formula (\ref{vz4.3}), we see that the four equalities in (\ref{vz4.4}) and the first formula in (\ref{vz4.5}) are true. So we only need to prove the last two equalities in (\ref{vz4.5}). We prove them for $k\in \Z_{+}$ by induction on $k$. For $-k\in \Z_{+}$, the arguments are similar. One has $\Theta_{k}(h) = h$ and $\Theta_{k}(f)= f$ when $k=0$. Suppose that the claim holds for $k>0$, then
\begin{align*}
   \overline{e}^{-k-1}f\overline{e}^{k+1}&=\overline{e}^{-1}\left(f-k\overline{e}^{-1}\overline{h}\right)\overline{e}= \overline{e}^{-1}f\overline{e}-k\overline{e}^{-2}\overline{h}\overline{e}
=f-(k+1)\overline{e}^{-1}\overline{h},  \\
   \overline{e}^{-k-1}h\overline{e}^{k+1}&=\overline{e}^{-1}\left(h+2k\right)\overline{e}=\overline{e}^{-1}h\overline{e}+2k=h+2(k+1).
\end{align*}
Hence, the claim follows and the proof is completed.\end{proof}
Motivated by Lemma \ref{vzlemma4.4}, we get the following proposition.
\begin{prop}For every $z\in \C$, there is a unique automorphism $\Theta_{z}$ of $U^{\left(\overline{e}\right)}$ such that
\begin{align*}
  \Theta_{z}(f) = f-z\overline{e}^{-1}\overline{h}, \ \ \ \
  \Theta_{z}(h) = h+2z,\ \ \ \
  \Theta_{z}(u)= u,
\end{align*}
where $u=\overline{f},\overline{h},\overline{e},e,\overline{e}^{-1}$, respectively. Moreover, we have $\Theta_{z}^{-1}=\Theta_{-z}$.
\end{prop}
Fix $z\in \C$, set $y\cdot u\cdot y'=\Theta_{z}(y)uy'$ for any $y,y'\in U(\g)$ and $u\in U^{\left(\overline{e}\right)}$.  This defines on $U^{\left(\overline{e}\right)}$ the structure  of a $U(\g)$-$U(\g)$-bimodule. Now we consider the functor $B_{z}:U(\g)$-mod $\longrightarrow$ $U(\g)$-mod of tensoring with the bimodule $U^{\left(\overline{e}\right)}$, that is
\begin{equation*}
B_{z}M=U^{\left(\overline{e}\right)}\bigotimes_{U(\g)}M,\ \ \ M\in U(\g)-\mbox{mod}.
\end{equation*}
The functors $B_{z}$ are called {\it Mathieu's twisting functors}. Now we are ready to prove Theorem \ref{th0}.

\noindent{\it Proof of Theorem \ref{th0}.} The action of $\overline{e}$ on $M_{\al,\b}^{\lambda,a,b}$ is,  clearly, bijective. Thus $M_{\al,\b}^{\lambda,a,b}$ carries the natural structure of a $U^{\left(\overline{e}\right)}$-module. Let $\phi:U^{\left(\overline{e}\right)}\bigotimes_{U(\g)}M_{\al,\b}^{\lambda,a,b}\longrightarrow M_{\al-2z,\b}^{\lambda,a,b}$ be the  linear map given by
\begin{equation*}
  \phi:1\otimes \eta_{\al_{k},\b^{s}}\longmapsto \eta_{\al_{k-z},\b^{s}}.
\end{equation*}
This is well defined and $\phi$ is an isomorphism of vector spaces. Next we show that $\phi$ is a homomorphism of $U(\g)$-modules. For any $\eta_{\al_{k},\b^{s}}\in M_{\al,\b}^{\lambda,a,b}$ with $k\in \Z$ and $s\in \N$, applying $\overline{e},e$ on $\phi\left(1\otimes \eta_{\al_{k},\b^{s}}\right)$, we obtain
\begin{eqnarray*}\!\!\!\!\!\!\!\!\!\!\!\!&\!\!\!\!\!\!\!\!\!\!\!\!\!\!\!&
\ \ \ \ \overline{e}\cdot \phi\left(1\otimes \eta_{\al_{k},\b^{s}}\right) =\overline{e} \cdot \eta_{\al_{k-z},\b^{s}}=-\lambda\eta_{\al_{k-z-1},\b^{s}},
\nonumber\\\!\!\!\!\!\!\!\!\!\!\!\!&\!\!\!\!\!\!\!\!\!\!\!\!\!\!\!&
\ \ \ \ \phi\left(\overline{e}\cdot 1\otimes \eta_{\al_{k},\b^{s}}\right)=\phi\left(1\otimes \overline{e}\cdot\eta_{\al_{k},\b^{s}}\right)
 =-\lambda\phi\left(1\otimes \eta_{\al_{k-1},\b^{s}}\right)=-\lambda\eta_{\al_{k-z-1},\b^{s}},
\nonumber\\\!\!\!\!\!\!\!\!\!\!\!\!&\!\!\!\!\!\!\!\!\!\!\!\!\!\!\!&
\ \ \ \ e\cdot \phi\left(1\otimes \eta_{\al_{k},\b^{s}}\right) =e \cdot \eta_{\al_{k-z},\b^{s}}=2\lambda\eta_{\al_{k-z-1},\b^{s+1}},
\nonumber\\\!\!\!\!\!\!\!\!\!\!\!\!&\!\!\!\!\!\!\!\!\!\!\!\!\!\!\!&
\ \ \ \ \phi\left(e\cdot 1\otimes \eta_{\al_{k},\b^{s}}\right)=\phi\left(1\otimes e\cdot\eta_{\al_{k},\b^{s}}\right)=2\lambda\phi\left(1\otimes \eta_{\al_{k-1},\b^{s+1}}\right)= 2\lambda\eta_{\al_{k-z-1},\b^{s+1}}.
\end{eqnarray*}
 Hence, we can conclude that $\phi\left(y\cdot 1\otimes \eta_{\al_{k},\b^{s}}\right)=y\cdot \phi\left(1\otimes \eta_{\al_{k},\b^{s}}\right)$ for $y=\overline{e},e$, respectively.

 Furthermore, letting $\overline{h},\overline{f}$ act on $\phi\left(1\otimes \eta_{\al_{k},\b^{s}}\right)$, respectively, we get
\begin{align*}
\overline{h}\cdot \phi\left(1\otimes \eta_{\al_{k},\b^{s}}\right)&=\overline{h} \cdot \eta_{\al_{k-z},\b^{s}}
      =-\b\eta_{\al_{k-z},\b^{s}}-(s-1)\eta_{\al_{k-z},\b^{s-1}},\\
\phi\left(\overline{h}\cdot 1\otimes \eta_{\al_{k},\b^{s}}\right)&=\phi\left(1\otimes \overline{h}\cdot\eta_{\al_{k},\b^{s}}\right)
 =-\b\eta_{\al_{k-z},\b^{s}}-(s-1)\eta_{\al_{k-z},\b^{s-1}},\\
 \overline{f}\cdot \phi\left(1\otimes \eta_{\al_{k},\b^{s}}\right)&=\overline{f} \cdot \eta_{\al_{k-z},\b^{s}}
     =\frac{1}{4\lambda}\left(K_{s}\eta_{\al_{k-z+1},\b^{s-2}}
+ 2K_{s}^{\b}\eta_{\al_{k-z+1},\b^{s-1}}+K_{a}^{\b}\eta_{\al_{k-z+1},\b^{s}}\right),\\
 \phi\left(\overline{f}\cdot 1\otimes \eta_{\al_{k},\b^{s}}\right)&=\phi\left(1\otimes \overline{f}\cdot\eta_{\al_{k},\b^{s}}\right)
 =\frac{1}{4\lambda}\left(K_{s}\eta_{\al_{k-z+1},\b^{s-2}}
+ 2K_{s}^{\b}\eta_{\al_{k-z+1},\b^{s-1}}+K_{a}^{\b}\eta_{\al_{k-z+1},\b^{s}}\right),
\end{align*}
where $K_{s}=(s-2)(s-1)$, $K_{s}^{\b}=(s-1)\b$ and $K_{a}^{\b}=\b^{2}+a$.
 Then it follows that $\phi\left(y\cdot 1\otimes \eta_{\al_{k},\b^{s}}\right)=y\cdot \phi\left(1\otimes \eta_{\al_{k},\b^{s}}\right)$, for $y=\overline{h}, \overline{f}$, respectively. Since $\Theta_{z}$ changes the generators $f$ and $h$ of $U(\g)$, applying $f$ and $h$ on $\phi\left(1\otimes \eta_{\al_{k},\b^{s}}\right)$, we have
\begin{eqnarray*}\!\!\!\!\!\!\!\!\!\!\!\!&\!\!\!\!\!\!\!\!\!\!\!\!\!\!\!&
\ \ \ \ \ \ \ \ f\cdot \phi\left(1\otimes \eta_{\al_{k},\b^{s}}\right) 
 =\frac{(s-1)(\al_{k-z}+s)}{2\lambda}\eta_{\al_{k-z+1},\b^{s-1}}
+ \frac{K_{a}^{\b}}{2\lambda}\eta_{\al_{k-z+1},\b^{s+1}}+\frac{\al_{k-z+s}\b+b}{2\lambda}\eta_{\al_{k-z+1},\b^{s}},
\nonumber\\\!\!\!\!\!\!\!\!\!\!\!\!&\!\!\!\!\!\!\!\!\!\!\!\!\!\!\!&
\ \ \ \ \ \ \ \ \phi\left(f\cdot 1\otimes \eta_{\al_{k},\b^{s}}\right)
   =\phi\left( 1\otimes f\cdot\eta_{\al_{k},\b^{s}}\right)-z\phi\left( 1\otimes \overline{e}^{-1}\overline{h}\cdot\eta_{\al_{k},\b^{s}}\right)
  \nonumber\\\!\!\!\!\!\!\!\!\!\!\!\!&\!\!\!\!\!\!\!\!\!\!\!\!\!\!\!&
\ \ \ \ \ \ \ \
=\frac{(s-1)(\al_{k}+s-2z)}{2\lambda} \eta_{\al_{k-z+1},\b^{s-1}} +\frac{K_{a}^{\b}}{2\lambda} \eta_{\al_{k-z+1},\b^{s+1}}
   +\frac{\al_{k+s}\b+b}{2\lambda} \eta_{\al_{k-z+1},\b^{s}}
   -\frac{z\b}{\lambda}\eta_{\al_{k-z+1},\b^{s}} 
\nonumber\\\!\!\!\!\!\!\!\!\!\!\!\!&\!\!\!\!\!\!\!\!\!\!\!\!\!\!\!&
\ \ \ \ \ \ \ \ =\frac{(s-1)(\al_{k-z}+s)}{2\lambda}\eta_{\al_{k-z+1},\b^{s-1}}+ \frac{K_{a}^{\b}}{2\lambda}\eta_{\al_{k-z+1},\b^{s+1}}
   +\frac{\al_{k-z+s}\b+b}{2\lambda}\eta_{\al_{k-z+1},\b^{s}},
\nonumber\\\!\!\!\!\!\!\!\!\!\!\!\!&\!\!\!\!\!\!\!\!\!\!\!\!\!\!\!&
\ \ \ \ \ \ \ \ h\cdot \phi\left(1\otimes \eta_{\al_{k},\b^{s}}\right)=h\cdot\eta_{\al_{k-z},\b^{s}}=-\al_{k-z}\eta_{\al_{k-z},\b^{s}},
\nonumber\\\!\!\!\!\!\!\!\!\!\!\!\!&\!\!\!\!\!\!\!\!\!\!\!\!\!\!\!&
\ \ \ \ \ \ \ \ \phi\left(h\cdot 1\otimes \eta_{\al_{k},\b^{s}}\right)= \phi\left(1\otimes (h+2z)\cdot \eta_{\al_{k},\b^{s}}\right)
=-\al_{k-z}\eta_{\al_{k-z},\b^{s}},
\end{eqnarray*}
where $K_{a}^{\b}=\b^{2}+a$. This shows that $\phi\left(y\cdot 1\otimes \eta_{\al_{k},\b^{s}}\right)=y\cdot \phi\left(1\otimes \eta_{\al_{k},\b^{s}}\right)$, for $y=f,h$, respectively. Thus, $\phi$ is a homomorphism of $U(\g)$-modules and the proof is completed.
\subsection{Weight modules associated to $\Theta(\lambda,a,b)$}
In this section, we will construct and classify a class of weight modules related to the non-weight module $\Theta(\lambda,a,b)$, which is defined in Definition $\ref{defi01}$. For $\al,\b, a, b \in \C$ and $\lambda\in \C^{\times}$, let $N_{\al,\b}^{\lambda,a,b}$ be a submodule of $\Theta(\lambda,a,b)^{*}$ with basis $\{\eta_{\al_{k},\b^{s}}\,|\,k\in \Z, s\in \N\}$.
The following proposition gives a precise action of $\g$ on $N_{\al,\b}^{\lambda,a,b}$.
\begin{prop}\label{defi3.21}For any $\al,\b,a, b\in \C$, $\lambda\in \C^{\times}$, $k\in \Z$ and $s\in \N$, $N_{\al,\b}^{\lambda,a,b}$ is a $\g$-module with the action defined as follows:
\begin{eqnarray}&\!\!\!\!\!\!\!\!\!\!\!\!\!\!\!\!\!\!\!\!\!\!\!\!&
  \overline{f}\cdot \eta_{\al_{k},\b^{s}}= -\lambda\eta_{\al_{k+1},\b^{s}}, \ \ \ \ \ \ \ \ \ \ \ \ \ \ \ \,
 f\cdot \eta_{\al_{k},\b^{s}}= -2\lambda\eta_{\al_{k+1},\b^{s+1}},\label{z4.22}\\
 &\!\!\!\!\!\!\!\!\!\!\!\!\!\!\!\!\!\!\!\!\!\!\!\!&
 h\cdot \eta_{\al_{k},\b^{s}} =-\al_{k}\eta_{\al_{k},\b^{s}},\ \ \ \ \ \ \ \ \ \ \ \ \ \ \ \ \,\overline{h}\cdot \eta_{\al_{k},\b^{s}} =-\b\eta_{\al_{k},\b^{s}}-(s-1)\eta_{\al_{k},\b^{s-1}},\label{vz4.23}\\
 &\!\!\!\!\!\!\!\!\!\!\!\!\!\!\!\!\!\!\!\!\!\!\!\!&
 \overline{e}\cdot \eta_{\al_{k},\b^{s}} = \frac{(s-2)(s-1)}{4\lambda}\eta_{\al_{k-1},\b^{s-2}}+\frac{(s-1)\b}{2\lambda}\eta_{\al_{k-1},\b^{s-1}}
 +\frac{\b^{2}+a}{4\lambda}\eta_{\al_{k-1},\b^{s}},\label{vz4.25}\\
 &\!\!\!\!\!\!\!\!\!\!\!\!\!\!\!\!\!\!\!\!\!\!\!\!&
 e\cdot\eta_{\al_{k},\b^{s}}= \frac{(s-1)(\al_{k}-s)}{2\lambda}\eta_{\al_{k-1},\b^{s-1}}+\frac{(\al_{k}-2s)\b+b}{2\lambda}\eta_{\al_{k-1},\b^{s}}
 -\frac{\b^{2}+a}{2\lambda}\eta_{\al_{k-1},\b^{s+1}}.\label{vz4.26}
\end{eqnarray}
\end{prop}
\begin{proof}The proof of this proposition is similar to Proposition \ref{defi3.1}, we omit it.\end{proof}
\begin{theo}\label{vztheo4.2}Let $N_{\al,\b}^{\lambda,a,b}$ be a $\g$-module defined in Proposition $\ref{defi3.21}$. Then the following holds:
\begin{itemize}
  \item[\rm(i)]   $N_{\al,\b}^{\lambda,a,b}$ is a simple $\g$-module if and only if $\b^{2}+a\neq0$ or $\b^{2}+a=0$ and $(\al_{k}-2s)\b+b\neq0$ for any $k\in \Z$ and $s\in \N$.
  \item[\rm(ii)] If $N_{\al,\b}^{\lambda,a,b}$ is not simple, then $N'_{k,s}:=U(\g)\eta_{\al_{k},\b}$ is the Verma module with the highest weight $\eta_{\al_{k},\b}$, where $k\in \Z$ such that $(\al_{k}-2s)\b+b=0$ with $s\in \N$. Moreover, $N_{\al,\b}^{\lambda,a,b}/N'_{k,s}\cong V'_{k,s}$, where $V'_{k,s}$ is a generalized weight $U(\g)$-module with infinite dimensional weight spaces.
  \item[\rm(iii)] For any $\al,\b,a,b\in \C$ and $\lambda,\lambda'\in \C^{\times}$, we have $N_{\al,\b}^{\lambda,a,b}\cong N_{\al,\b}^{\lambda',a,b}$.
  \item[\rm(iv)] For any $\al,\b,\al',\b',a,b,a',b'\in \C$ and $\lambda,\lambda'\in \C^{\times}$, we have $N_{\al,\b}^{\lambda,a,b}\cong N_{\al',\b'}^{\lambda',a',b'}$ if and only if $\al-\al'\in 2\Z$, $\b=\b'$, $a=a'$ and $b=b'$.
\end{itemize}\end{theo}
\begin{proof}The proof is similar to  that of Theorem \ref{vztheo4.1}, we omit the details.
\end{proof}
\begin{remark}
\begin{itemize}
  \item[\rm(1)]\rm{ Due to Theorem \ref{mzh4} (i), we observe that there is a close connection between the $\g$-modules  $N_{\al,\b}^{\lambda,a,b}$ and $M_{\al,\b}^{\lambda,a,b}$.
  \item[\rm(2)]
       We can also use Mathieu's twisting functors $B'_{z}$ with $z\in \C$ and show that $B'_{z}N_{\al,\b}^{\lambda,a,b}=N_{\al+2z,\b}^{\lambda,a,b}$. The construction of $B'_{z}$ is similar to $B_{z}$, which is constructed in Section 4.1.2 and we leave the detail to readers.}
\end{itemize}

\end{remark}
\subsection{Weight modules associated to $\Omega\left(\lambda,a,\beta_{1}\left(\overline{h}\right)\right)$}
In this section, we will construct and investigate a class of weight modules related to the non-weight module $\Omega\left(\lambda,a,\beta_{1}\left(\overline{h}\right)\right)$, which is defined in Definition $\ref{defi01}$. For $\al,\b,a\in \C$, $\lambda\in \C^{\times}$, $m\in \Z_{+}$ and $\beta_{1}\left(\overline{h}\right)=\sum_{r=0}^{m}q_{r}\overline{h}^{r}\in \C\left[\,\overline{h}\,\right]$ with $q_{r}\in \C$, let $V_{\al,\b}^{\lambda,a,\beta_{1}\left(\overline{h}\right)}$ be the submodule of $\Omega\left(\lambda,a,\beta_{1}\left(\overline{h}\right)\right)^{*}$, spanned by the linear independent collection $\{\eta_{\al_{k},\b^{s}}\,|\,k\in \Z, s\in \N\}$.
Then the following proposition gives a precise action of $\g$ on $V_{\al,\b}^{\lambda,a,\beta_{1}\left(\overline{h}\right)}$.
\begin{prop}\label{prop4.9}Let $k\in \Z$, $s\in \N$ and $V_{\al,\b}^{\lambda,a,\beta_{1}\left(\overline{h}\right)}$ be the $\g$-module defined above. Then the action of $\g$ on this module is  as follows:
\begin{eqnarray}&\!\!\!\!\!\!\!\!\!\!\!\!\!\!\!\!\!\!\!\!\!\!\!\!&
 \overline{h}\cdot \eta_{\al_{k},\b^{s}} =-\b\eta_{\al_{k},\b^{s}}-(s-1)\eta_{\al_{k},\b^{s-1}},\ \ \ \ \ \ \
  h\cdot \eta_{\al_{k},\b^{s}} =-\al_{k}\eta_{\al_{k},\b^{s}},\label{xy4.17}\\
 &\!\!\!\!\!\!\!\!\!\!\!\!\!\!\!\!\!\!\!\!\!\!\!\!&
  \overline{f}\cdot \eta_{\al_{k},\b^{s}}= \frac{\b-a}{2\lambda}\eta_{\al_{k+1},\b^{s}}+\frac{s-1}{2\lambda}\eta_{\al_{k+1},\b^{s-1}},\label{xy4.18} \\
  &\!\!\!\!\!\!\!\!\!\!\!\!\!\!\!\!\!\!\!\!\!\!\!\!&
 \overline{e}\cdot \eta_{\al_{k},\b^{s}}= -\frac{\lambda(\b+a)}{2}\eta_{\al_{k-1},\b^{s}}-\frac{\lambda(s-1)}{2}\eta_{\al_{k-1},\b^{s-1}},\label{xy4.19} \\
 &\!\!\!\!\!\!\!\!\!\!\!\!\!\!\!\!\!\!\!\!\!\!\!\!&
e\cdot \eta_{\al_{k},\b^{s}} = -\frac{\lambda\al_{k-s+1}+2\al_{1}(\b)}{2}\eta_{\al_{k-1},\b^{s}}+\lambda(\b+a)\eta_{\al_{k-1},\b^{s+1}}
 \nonumber\\\!\!\!\!\!\!\!\!\!\!\!\!&\!\!\!\!\!\!\!\!\!\!\!\!\!\!\!&
 \ \ \ \ \ \ \ \ \ \ \ \ \ \ \,
 -\SUM{\substack{l=1}}{m}\SUM{\substack{i=0\\ i\neq l}}{l}\frac{p_{l}\binom{l}{i}\b^{i}(s-1)!}{(s-l+i-1)!}\eta_{\al_{k-1},\b^{s-l+i}},\label{xy4.20}\\
 &\!\!\!\!\!\!\!\!\!\!\!\!\!\!\!\!\!\!\!\!\!\!\!\!&
 f\cdot\eta_{\al_{k},\b^{s}}= \frac{\al_{k+s-1}-2\lambda\b_{1}(\b)}{2\lambda}\eta_{\al_{k+1},\b^{s}}+\frac{\b-a}{\lambda}\eta_{\al_{k+1},\b^{s+1}}
 \nonumber\\\!\!\!\!\!\!\!\!\!\!\!\!&\!\!\!\!\!\!\!\!\!\!\!\!\!\!\!&
  \ \ \ \ \ \ \ \ \ \ \ \ \ \ \,-\SUM{\substack{r=1}}{m}\SUM{\substack{j=0\\ j\neq r}}{r}\frac{q_{r}\binom{r}{j}\b^{j}(s-1)!}{(s-r+j-1)!}\eta_{\al_{k+1},\b^{s-r+j}},\label{xy4.21}
\end{eqnarray}
where $\b_{1}\left(\b\right)=\sum_{r=0}^{m}q_{r}\b^{r}$ and $\al_{1}\left(\b\right)=\sum_{l=0}^{m}p_{l}\b^{l}$ with $q_{r},p_{l}\in \C$ and satisfy (\ref{1}).\end{prop}
\begin{proof}The proof of this proposition is similar to that of Proposition \ref{defi3.1}, we omit it.\end{proof}
\begin{theo}\label{theo4.7}Let $V_{\al,\b}^{\lambda,a,\beta_{1}\left(\overline{h}\right)}$ be a $\g$-module defined in Proposition \ref{prop4.9}. Then the following holds:
\begin{itemize}
  \item[\rm(i)] $V_{\al,\b}^{\lambda,a,\beta_{1}\left(\overline{h}\right)}$ is a simple $\g$-module if and only if one of the following conditions is satisfied:
      \begin{itemize}
      \item[\rm(1)]$\b\neq \pm a$,
      \item[\rm(2)] $\b=a$, $a\neq0$ and $\al_{k+s-1}-2\lambda \b_{1}\left(\b\right)\neq 0$ for any $k\in \Z$ and $s\in \N$,
      \item[\rm(3)] $\b=-a$, $a\neq0$ and $\lambda\al_{k-s+1}+2 \al_{1}\left(\b\right)\neq 0$ for any $k\in \Z$ and $s\in \N$.
      \end{itemize}
  \item[\rm(ii)] For any $\lambda\in \C^{\times}$, $\al,\b,\al',\b',a,a'\in \C$ and $\beta_{1}\left(\overline{h}\right),\beta'_{1}\left(\overline{h}\right)\in \C\left[\,\overline{h}\,\right]$, we have $V_{\al,\b}^{\lambda,a,\beta_{1}\left(\overline{h}\right)}\cong V_{\al',\b'}^{\lambda,a',\beta'_{1}\left(\overline{h}\right)}$ if and only if $\al-\al'\in 2\Z$, $\b=\b'$, $a=a'$ and $\beta_{1}\left(\overline{h}\right)=\beta'_{1}\left(\overline{h}\right)$.
\end{itemize}\end{theo}
\begin{proof}The proof of statement (i) is similar to that of Theorem \ref{vztheo4.1}, we omit the detail.

Now we prove statement (ii). The sufficiency of the conditions is clear. Suppose that $\varphi:V_{\al,\b}^{\lambda,a,\beta_{1}\left(\overline{h}\right)}\longrightarrow V_{\al',\b'}^{\lambda,a',\beta_{1}'\left(\overline{h}\right)}$ is the isomorphism of $U(\g)$-modules. Let $\eta_{\al_{k},\b^{s}}\in V_{\al,\b}^{\lambda,a,\beta_{1}\left(\overline{h}\right)}$ with $k\in \Z$ and $s\in \N$. Then, by a similar proof of Theorem \ref{vztheo4.1} (iv), we obtain that 
\begin{align}\label{xy4.23}
\al-\al'\in 2\Z,\ \ \ \ \b=\b'\ \ \ \ \mbox{and}\ \ \ \ \varphi\left(\eta_{\al_{k},\b}\right)=d_{k}\eta_{\al_{k},\b} \ \ \ \mbox{for\ some}\ \ d_{k}\in \C^{\times}.
\end{align}
Applying $\overline{e}$ and $\overline{f}$ on $\varphi\left(\eta_{\al_{k},\b^{s}}\right)$, respectively, from (\ref{xy4.18}) and (\ref{xy4.19}), we get
\begin{align*}
\overline{e}\cdot \varphi\left(\eta_{\al_{k},\b}\right)&=-\frac{\lambda\left(\b+a'\right)}{2}d_{k}\eta_{\al_{k-1},\b},\ \ \ \
\varphi\left(\overline{e}\cdot \eta_{\al_{k},\b}\right)=-\frac{\lambda(\b+a)}{2}d_{k-1}\eta_{\al_{k-1},\b},\\
\overline{f}\cdot \varphi\left(\eta_{\al_{k},\b}\right)&=\frac{\b-a'}{2\lambda}d_{k}\eta_{\al_{k+1},\b},\ \ \ \ \ \ \ \ \ \
\varphi\left(\overline{f}\cdot \eta_{\al_{k},\b}\right)=\frac{\b-a}{2\lambda}d_{k+1}\eta_{\al_{k+1},\b}.
\end{align*}
Hence, we have \begin{align}\label{xy4.22}
\left(\b+a'\right)d_{k}=\left(\b+a\right)d_{k-1},\ \ \ \ \left(\b-a'\right)d_{k}=\left(\b-a\right)d_{k+1},
                   \end{align}
 which implies that  $\left(a'\right)^{2}=a^{2}$.

 In fact, $a'=a$. Otherwise, due to (\ref{xy4.22}), we can assume $\b+a\neq0$ and $a'=-a\neq0$. Then, considering the actions of $e,f$ on $\varphi\left(\eta_{\al_{k},\b^{s}}\right)$, according to (\ref{xy4.20}) and (\ref{xy4.21}), we have
 \begin{align*}
   e\cdot \varphi\left(\eta_{\al_{k},\b}\right)&=-\left(\frac{\lambda\al_{k}}{2}+\al_{1}'(\b)\right)d_{k}\eta_{\al_{k-1},\b}+ \lambda\left(\b+a'\right)d_{k}\eta_{\al_{k-1},\b^{2}},\\
 \varphi\left(e\cdot\eta_{\al_{k},\b}\right)&=-\left(\frac{\lambda\al_{k}}{2}+\al_{1}(\b)\right)d_{k-1}\eta_{\al_{k-1},\b}+ \lambda\left(\b+a\right)\varphi\left(\eta_{\al_{k-1},\b^{2}}\right),\\
  f\cdot \varphi\left(\eta_{\al_{k},\b}\right)&=\left(\frac{\al_{k}}{2\lambda}-\b_{1}'(\b)\right)d_{k}\eta_{\al_{k+1},\b}+ \frac{\left(\b-a'\right)}{\lambda}d_{k}\eta_{\al_{k+1},\b^{2}},\\
 \varphi\left(f\cdot\eta_{\al_{k},\b}\right)&=\left(\frac{\al_{k}}{2\lambda}-\b_{1}(\b)\right)d_{k+1}\eta_{\al_{k+1},\b}+ \frac{\left(\b-a\right)}{\lambda}\varphi\left(\eta_{\al_{k+1},\b^{2}}\right).
 \end{align*}
Then, it follows that \begin{align*}
                       \left(\frac{\al_{k-1}}{2\lambda}-\b_{1}(\b)\right)d_{k}
                       +\left(\frac{\al_{k+1}}{2\lambda}+\frac{\al_{1}(\b)}{\lambda^{2}}\right)d_{k-1}
                       -\left(\frac{\al_{k+1}}{2\lambda}+\frac{\al_{1}'(\b)}{\lambda^{2}}\right)d_{k}
                       =\left(\frac{\al_{k-1}}{2\lambda}-\b_{1}'(\b)\right)d_{k-1}.
                      \end{align*}
Take $k$ large enough, then we have $\frac{\al_{k-1}}{2\lambda}d_{k}+\frac{\al_{k+1}}{2\lambda}d_{k-1}-\frac{\al_{k+1}}{2\lambda}d_{k}=\frac{\al_{k-1}}{2\lambda}d_{k-1}$. Hence, from (\ref{xy4.22}), we have $a=0$, a contradiction. Hence $a'=a.$

Therefore, it remains to prove that $\beta_{1}\left(\overline{h}\right)=\beta'_{1}\left(\overline{h}\right)$. Write $\beta_{1}\left(\overline{h}\right)=\sum_{l=0}^{n}q_{l}\overline{h}^{l}$ and $\beta_{1}'\left(\overline{h}\right)=\sum_{r=0}^{m}q_{l}'\overline{h}^{r}$.  Then, we consider the following cases.

\noindent{\bf Case 1:} $a\neq0$. From (\ref{xy4.22}), we have $d_{k}=d_{k-1}$ for any $k\in \Z$. Without loss of generality, we assume $d_{k}=1$ for any $k\in \Z$.  Then we get the following claim.
\begin{clai}$\varphi\left(\eta_{\al_{k},\b^{s}}\right)=\eta_{\al_{k},\b^{s}}$ for any $k\in \Z$ and $s\in \N$.\end{clai}
We prove this claim by induction on $s$. According to (\ref{xy4.23}), we see that Claim 2 holds for $s=1$. Suppose that Claim 2 is true for $s>1$. For any $k\in \Z$, write $\varphi\left(\eta_{\al_{k},\b^{s+1}}\right)=\sum_{l=1}^{s+1}c_{k,l}\eta_{\al_{k},\b^{l}}$. Then, due to (\ref{xy4.17}), we obtain that
\begin{align*}
  \overline{h}\cdot\varphi\left(\eta_{\al_{k},\b^{s+1}}\right)&=-\b\varphi\left(\eta_{\al_{k},\b^{s
  +1}}\right)
-\SUM{\substack{l=1}}{s+1}c_{k,l}(l-1)\eta_{\al_{k},\b^{l-1}}, \\
 \varphi\left(\overline{h}\cdot\eta_{\al_{k},\b^{s+1}}\right)&=-\b\varphi\left(\eta_{\al_{k},\b^{s+1}}\right)-s\eta_{\al_{k},\b^{s}},
\end{align*}
which implies that $\varphi\left(\eta_{\al_{k},\b^{s+1}}\right)=\eta_{\al_{k},\b^{s+1}}$. Hence, the proof of the claim is completed.

Let $s\in \N$ and $s>m$. Then, applying $f$ on $\varphi\left(\eta_{\al_{k},\b^{s}}\right)$, we have
\begin{align*}
  f\cdot\varphi\left(\eta_{\al_{k},\b^{s}}\right)&=\frac{\al_{k+s-1}-2\lambda\b_{1}'(\b)}{2\lambda}\eta_{\al_{k+1},\b^{s}}
  +\frac{\b-a}{\lambda}\eta_{\al_{k+1},\b^{s+1}}-\SUM{r=1}{m}\SUM{\substack{j=0\\ j\neq r}}{r}\frac{q_{r}'\binom{r}{j}\b^{j}(s-1)!}{(s-r+j-1)!}\eta_{\al_{k+1},\b^{s-r+j}},  \\
 \varphi\left( f\cdot\eta_{\al_{k},\b^{s}}\right)&=\frac{\al_{k+s-1}-2\lambda\b_{1}(\b)}{2\lambda}\eta_{\al_{k+1},\b^{s}}
 +\frac{\b-a}{\lambda}\eta_{\al_{k+1},\b^{s+1}}-\SUM{l=1}{n}\SUM{\substack{i=0\\ i\neq l}}{l}\frac{q_{l}\binom{l}{i}\b^{i}(s-1)!}{(s-l+i-1)!}\eta_{\al_{k+1},\b^{s-l+i}}.
\end{align*}
Consequently, we get $\b_{1}'(\b)=\b_{1}(\b)$ and
\begin{align*}
 \SUM{r=1}{m}\SUM{\substack{j=0\\ j\neq r}}{r}\frac{q_{r}'\binom{r}{j}\b^{j}(s-1)!}{(s-r+j-1)!}\eta_{\al_{k+1},\b^{s-r+j}}=\SUM{l=1}{n}\SUM{\substack{i=0\\ i\neq l}}{l}\frac{q_{l}\binom{l}{i}\b^{i}(s-1)!}{(s-l+i-1)!} \eta_{\al_{k+1},\b^{s-l+i}}.
\end{align*}
Then, comparing the coefficients of $\eta_{\al_{k+1},\b^{s-r}}$ for $r\in \{1,2,\cdots,m\}$ in above equation, we obtain that $m=n$ and $q_{r}'=q_{r}$ for $r\in \{1,2,\cdots,m\}$. Thus, $q_{0}'=q_{0}$ and $\b_{1}'\left(\overline{h}\right)=\b_{1}\left(\overline{h}\right)$.

\noindent{\bf Case 2:} $a=0$. If $\b\neq0$, due to (\ref{xy4.22}), we have $d_{k}=d_{k-1}$ for any $k\in \Z$. If $\b=0$, we write $\varphi\left(\eta_{\al_{k},\b^{2}}\right)=c_{k,1}\eta_{\al_{k},\b}+c_{k,2}\eta_{\al_{k},\b^{2}}$ for some $c_{k,1}\in \C$ and $c_{k,2}\in \C^{\times}$. Acting by $\overline{e}$ and $\overline{h}$, respectively, on $\varphi\left(\eta_{\al_{k},\b^{2}}\right)$, we obtain
\begin{align*}
  \overline{e}\cdot\varphi\left(\eta_{\al_{k},\b^{2}}\right)&=-\frac{\lambda}{2}c_{k,2}\eta_{\al_{k-1},\b}, \ \ \ \ \varphi\left(\overline{e}\cdot\eta_{\al_{k},\b^{2}}\right)=-\frac{\lambda}{2}d_{k-1}\eta_{\al_{k-1},\b}, \\
  \overline{h}\cdot\varphi\left(\eta_{\al_{k},\b^{2}}\right)&=-c_{k,2}\eta_{\al_{k},\b}, \ \ \ \ \ \ \ \ \, \varphi\left(\overline{h}\cdot\eta_{\al_{k},\b^{2}}\right)=-d_{k}\eta_{\al_{k},\b},
\end{align*}
which implies that $d_{k-1}=c_{k,2}=d_{k}$ for any $k\in \Z$. Then, by a similar proof in Case 1, we get $\b_{1}'\left(\overline{h}\right)=\b_{1}\left(\overline{h}\right)$.

Therefore, according to above two cases, we complete the proof.
\end{proof}
\begin{remark}\rm{Theorem \ref{mzh4} {\rm(ii)} suggests a close relation between the $\g$-modules $V_{\al,\b}^{\lambda,a,\beta_{1}\left(\overline{h}\right)}$ and $M_{\al,\b}^{\lambda,a,b}$.}
\end{remark}
\subsection{Proof of Theorem \ref{mzh4}}
Here we just give the proof of statement (ii) of Theorem \ref{mzh4}. The proof of statement (i) is similar. Now we prove statement (ii). Let $\varphi':M_{\al,\b}^{\lambda,-a^{2},b}\longrightarrow V_{\al,\b}^{\lambda,a,\b_{1}\left(\overline{h}\right)}$ be the linear map defined by
\begin{equation*}
  \varphi':\eta_{\al_{k},\b^{s}}\longmapsto \frac{c^{k+s-1}}{(2\lambda)^{s-1}}e^{s-1}\cdot\eta_{\al_{k+s-1},\b},
\end{equation*}
where $c=\frac{2}{\b+a}$. The map $\varphi'$ is well defined and bijective. We claim that $\varphi'$ is a homomorphism of $U(\g)$-modules. Applying $\overline{e},e$ on $\varphi'\left(\eta_{\al_{k},\b^{s}}\right)$, respectively, and using (\ref{vz4.11}), (\ref{xy4.19}) and (\ref{xy4.20}), we obtain
\begin{align*}
\varphi'\left(\overline{e}\cdot\eta_{\al_{k},\b^{s}}\right)&=-\lambda\varphi'\left(\eta_{\al_{k-1},\b^{s}}\right)
=-\frac{\lambda c^{k+s-2}}{(2\lambda)^{s-1}}e^{s-1}\cdot\eta_{\al_{k+s-2},\b},\\
\overline{e}\cdot \varphi'\left(\eta_{\al_{k},\b^{s}}\right)&=\frac{c^{k+s-1}}{(2\lambda)^{s-1}}\overline{e} e^{s-1}\cdot\eta_{\al_{k+s-1},\b}
=-\frac{\lambda c^{k+s-2}}{(2\lambda)^{s-1}}e^{s-1}\cdot\eta_{\al_{k+s-2},\b},\\
\varphi'\left(e\cdot\eta_{\al_{k},\b^{s}}\right)&=2\lambda\varphi'\left(\eta_{\al_{k-1},\b^{s+1}}\right)=\frac{c^{k+s-1}}{(2\lambda)^{s-1}} e^{s}\cdot\eta_{\al_{k+s-1},\b},\\
e\cdot\varphi'\left(\eta_{\al_{k},\b^{s}}\right)&
=\frac{c^{k+s-1}}{(2\lambda)^{s-1}} e^{s}\cdot\eta_{\al_{k+s-1},\b},
\end{align*}
which implies that $\varphi'\left(y\cdot\eta_{\al_{k},\b^{s}}\right)=y\cdot\varphi'\left(\eta_{\al_{k},\b^{s}}\right)$, for $y=\overline{e},e$, respectively. Next, applying $h$ and $\overline{h}$ to $\varphi'\left(\eta_{\al_{k},\b^{s}}\right)$, respectively, and using (\ref{vz14.4}) and (\ref{xy4.17}), we have
\begin{align*}
\varphi'\left(h\cdot\eta_{\al_{k},\b^{s}}\right)&=-\al_{k}\varphi'\left(\eta_{\al_{k},\b^{s}}\right)=-\frac{\al_{k}c^{k+s-1}}{(2\lambda)^{s-1}} e^{s-1}\cdot\eta_{\al_{k+s-1},\b},\\
h\cdot\varphi'\left(\eta_{\al_{k},\b^{s}}\right)&=\frac{c^{k+s-1}}{(2\lambda)^{s-1}}h e^{s-1}\cdot\eta_{\al_{k+s-1},\b}=\frac{c^{k+s-1}}{(2\lambda)^{s-1}}e^{s-1}(h+2(s-1))\cdot\eta_{\al_{k+s-1},\b}\\
&=-\frac{\al_{k}c^{k+s-1}}{(2\lambda)^{s-1}} e^{s-1}\cdot\eta_{\al_{k+s-1},\b},\\
\varphi'\left(\overline{h}\cdot\eta_{\al_{k},\b^{s}}\right)&=-\b\varphi'\left(\eta_{\al_{k},\b^{s}}\right)
-(s-1)\varphi'\left(\eta_{\al_{k},\b^{s-1}}\right)\\
&=-\frac{\b c^{k+s-1}}{(2\lambda)^{s-1}}e^{s-1}\cdot\eta_{\al_{k+s-1},\b}-\frac{(s-1) c^{k+s-2}}{(2\lambda)^{s-2}}e^{s-2}\cdot\eta_{\al_{k+s-2},\b},\\
\overline{h}\cdot\varphi'\left(\eta_{\al_{k},\b^{s}}\right)&=\frac{c^{k+s-1}}{(2\lambda)^{s-1}}\overline{h}e^{s-1}\cdot\eta_{\al_{k+s-1},\b}
=\frac{c^{k+s-1}}{(2\lambda)^{s-1}}\left(2(s-1)e^{s-2}\overline{e}+e^{s-1}\overline{h}\right)\cdot\eta_{\al_{k+s-1},\b}\\
&=-\frac{\b c^{k+s-1}}{(2\lambda)^{s-1}}e^{s-1}\cdot\eta_{\al_{k+s-1},\b}-\frac{(s-1)c^{k+s-2}}{(2\lambda)^{s-2}}e^{s-2}\cdot\eta_{\al_{k+s-2},\b}.
\end{align*}
This implies that $\varphi'\left(y\cdot\eta_{\al_{k},\b^{s}}\right)=y\cdot\varphi'\left(\eta_{\al_{k},\b^{s}}\right)$ for $y=h,\overline{h}$, respectively. Moreover, letting $\overline{f}$ and $f$ act on $\varphi'\left(\eta_{\al_{k},\b^{s}}\right)$ and using (\ref{vz14.5}), (\ref{vz14.6}),  (\ref{xy4.18}) and (\ref{xy4.21}), we have that
\begin{eqnarray*}\!\!\!\!\!\!\!\!\!\!\!\!&\!\!\!\!\!\!\!\!\!\!\!\!\!\!\!&
\ \ \ \ \ \ \ \ \ \varphi'\left(\overline{f}\cdot\eta_{\al_{k},\b^{s}}\right)=\frac{K_{s}}{4\lambda}\varphi'\left(\eta_{\al_{k+1},\b^{s-2}}\right)
+\frac{(s-1)\b}{2\lambda}\varphi'\left(\eta_{\al_{k+1},\b^{s-1}}\right)+\frac{\b^{2}-a^{2}}{4\lambda}\varphi'\left(\eta_{\al_{k+1},\b^{s}}\right)
\nonumber\\\!\!\!\!\!\!\!\!\!\!\!\!&\!\!\!\!\!\!\!\!\!\!\!\!\!\!\!&
\ \ \ \ \ \ \ =\frac{c^{k+s-2}}{(2\lambda)^{s-3}}\left(\frac{K_{s}}{4\lambda}e^{s-3}\cdot\eta_{\al_{k+s-2},\b}
+\frac{(s-1)\b c}{4\lambda^{2}}e^{s-2}\cdot\eta_{\al_{k+s-1},\b}\right)
+\frac{c^{k+s-1}}{(2\lambda)^{s}}(\b-a)e^{s-1}\cdot\eta_{\al_{k+s},\b},
\nonumber\\\!\!\!\!\!\!\!\!\!\!\!\!&\!\!\!\!\!\!\!\!\!\!\!\!\!\!\!&
\ \ \ \ \ \ \ \ \ \overline{f}\cdot\varphi'\left(\eta_{\al_{k},\b^{s}}\right)=\frac{c^{k+s-1}}{(2\lambda)^{s-1}}\overline{f}e^{s-1}\cdot\eta_{\al_{k+s-1},\b} =-\frac{c^{k+s-1}}{(2\lambda)^{s-1}}\left(K_{s}e^{s-3}\overline{e}+(s-1)e^{s-2}\overline{h}-e^{s-1}\overline{f}\right)\cdot\eta_{\al_{k+s-1},\b}
\nonumber\\\!\!\!\!\!\!\!\!\!\!\!\!&\!\!\!\!\!\!\!\!\!\!\!\!\!\!\!&
\ \ \ \ \ \ \ =\frac{c^{k+s-2}}{(2\lambda)^{s-3}}\left(\frac{K_{s}}{4\lambda}e^{s-3}\cdot\eta_{\al_{k+s-2},\b}
+\frac{(s-1)\b c}{4\lambda^{2}}e^{s-2}\cdot\eta_{\al_{k+s-1},\b}\right)
+\frac{c^{k+s-1}}{(2\lambda)^{s}}(\b-a)e^{s-1}\cdot\eta_{\al_{k+s},\b},
\nonumber\\\!\!\!\!\!\!\!\!\!\!\!\!&\!\!\!\!\!\!\!\!\!\!\!\!\!\!\!&
\ \ \ \ \ \ \ \ \ \varphi'\left(f\cdot\eta_{\al_{k},\b^{s}}\right)=\frac{(s-1)(\al_{k}+s)}{2\lambda}\varphi'\left(\eta_{\al_{k+1},\b^{s-1}}\right)
+\frac{\b^{2}-a^{2}}{2\lambda}\varphi'\left(\eta_{\al_{k+1},\b^{s+1}}\right)
+\frac{\al_{k+s}\b+b}{2\lambda}\varphi'\left(\eta_{\al_{k+1},\b^{s}}\right)
\nonumber\\\!\!\!\!\!\!\!\!\!\!\!\!&\!\!\!\!\!\!\!\!\!\!\!\!\!\!\!&
\ \ \ \ \ \ \ =\frac{(s-1)(\al_{k}+s)}{2\lambda}\frac{c^{k+s-1}}{(2\lambda)^{s-2}}e^{s-2}\cdot\eta_{\al_{k+s-1},\b}
+\frac{\al_{k+s}\b+b}{2\lambda}\frac{c^{k+s}}{(2\lambda)^{s-1}}e^{s-1}\cdot\eta_{\al_{k+s},\b}
+\frac{c^{k+s}}{(2\lambda)^{s}}\frac{\b-a}{\lambda}e^{s}\cdot\eta_{\al_{k+s+1},\b},
\nonumber\\\!\!\!\!\!\!\!\!\!\!\!\!&\!\!\!\!\!\!\!\!\!\!\!\!\!\!\!&
\ \ \ \ \ \ \ \ \ f\cdot\varphi'\left(\eta_{\al_{k},\b^{s}}\right)=\frac{c^{k+s-1}}{(2\lambda)^{s-1}}fe^{s-1}\cdot\eta_{\al_{k+s-1},\b}
=-\frac{c^{k+s-1}}{(2\lambda)^{s-1}}\left((s-1)e^{s-2}(h+s-2)-e^{s-1}f\right)\cdot\eta_{\al_{k+s-1},\b}
\nonumber\\\!\!\!\!\!\!\!\!\!\!\!\!&\!\!\!\!\!\!\!\!\!\!\!\!\!\!\!&
\ \ \ \ \ \ \ =\frac{c^{k+s-1}}{(2\lambda)^{s-2}}\left(\frac{(s-1)(\al_{k}+s)}{2\lambda}e^{s-2}\cdot\eta_{\al_{k+s-1},\b}
+\frac{\al_{k+s-1}-2\lambda\b_{1}\left(\b\right)}{(2\lambda)^{2}}e^{s-1}\cdot\eta_{\al_{k+s},\b}
+\frac{2(\b-a)}{(2\lambda)^{2}}e^{s-1}\cdot\eta_{\al_{k+s},\b^{2}}\right)
\nonumber\\\!\!\!\!\!\!\!\!\!\!\!\!&\!\!\!\!\!\!\!\!\!\!\!\!\!\!\!&
\ \ \ \ \ \ \ =\frac{(s-1)(\al_{k}+s)}{2\lambda}\frac{c^{k+s-1}}{(2\lambda)^{s-2}}e^{s-2}\cdot\eta_{\al_{k+s-1},\b}
+\frac{c^{k+s}}{(2\lambda)^{s}}\frac{\b-a}{\lambda}e^{s}\cdot\eta_{\al_{k+s+1},\b}
+\frac{\al_{k+s}\b+P_{a}^{\b}}{(2\lambda)}\frac{c^{k+s}}{(2\lambda)^{s-1}}e^{s-1}\cdot\eta_{\al_{k+s},\b},
\end{eqnarray*}
where $K_{s}=(s-1)(s-2)$ and $P_{a}^{\b}=\frac{\b-a}{\lambda}\al_{1}\left(\b\right)-\lambda(\b+a)\b_{1}\left(\b\right)-2a$. Note that $b=-2 a\left(\lambda\b_{1}(a)+1\right)$. Then by (\ref{1}), we get $P_{a}^{\b}=b$. Then, it follows that $\varphi'\left(y\cdot\eta_{\al_{k},\b^{s}}\right)=y\cdot\varphi'\left(\eta_{\al_{k},\b^{s}}\right)$, for $y=\overline{f},f$, respectively. Thus $\varphi'$ is a homomorphism of $U(\g)$-modules and the proof of the theorem is completed.
\section*{Acknowledgments}
The author is very grateful to Volodymyr Mazorchuk for his many helpful discussions and ideas during this project.


\begin{thebibliography}{9999}\vskip0pt\small
\parindent=2ex\parskip=-1.5pt\baselineskip=-1.5pt\lineskip=3.0pt
\def\re#1{\bibitem{#1}\label{#1}}
\re{BM}P. Batra, V. Mazorchuk, Blocks and modules for Whittaker pairs, {\em J. Pure Appl. Algebra} {\bf215(7)} (2011), 1552--1568.
\re{RB}R. Block, The irreducible representations of the Lie algebra $\mathfrak{sl}(2)$ and of the Weyl algebra, {\em Adv. Math.} {\bf39 (1)} (1981), 69--110.
\re{EC}E. Cartan, Les groupes projectifs qui ne laissent invariante aucune multiplicit\'{e} plane, {\em Bull. Soc. Math. France} {\bf41} (1913) 53--96.
\re{MC}M. Chaffe, Category $\mathcal{O}$ for Takiff Lie algebras, Preprint arXiv:2205.03121.
\re{CC}Q. Chen, Y. Cai, Modules over algebras related to the Virasoro algebra, {\em Internat. J. Math.} {\bf26} (2015), 1550070.
\re{CG} H. Chen, X. Guo, Non-weight modules over the Heisenberg-Virasoro algebra and the $W$ algebra $W(2, 2)$, {\em J. Algebra Appl.} {\bf16} (2017), 1750097.
\re{CG1}H. Chen, X. Guo, A new family of modules over the Virasoro algebra, {\em J. Algebra} {\bf457} (2016), 73--105.
\re{CY} Q. Chen, Y. Yao, Non-weight modules over algebras related to the Virasoro algebra, {\em J. Geom. Phys.} {\bf134} (2018), 11--18.
\re{CCM}C.-W. Chen, K. Coulembier, V. Mazorchuk, Translated simple modules for Lie algebras and simple supermodules for Lie superalgebras, {\em Mathematische Zeitschrift} {\bf 297} (2021), 255--281.
\re{CM}C.-W. Chen, V. Mazorchuk, Simple supermodules over Lie superalgebras, {\em Trans. Amer. Math. Soc.} {\bf374} (2021), 899--921.
\re{JD}J. Dixmier, Enveloping Algebras, American Mathematical Society, 1977.
\re{DFO} Yu. Drozd, V. Futorny, S. Ovsienko. Harish-Chandra subalgebras and Gelfand-Zetlin modules.
Finite-dimensional algebras and related topics (Ottawa, ON, 1992), 79--93, NATO Adv. Sci.
Inst. Ser. C: Math. Phys. Sci., {\bf 424}, Kluwer Acad. Publ., Dordrecht, 1994.
\re{EMV} N. Early, V. Mazorchuk, E. Vishnyakova, Canonical Gelfand-Zeitlin modules over orthogonal Gelfand-Zeitlin algebras, {\em International Mathematics Research Notices} {\bf 2020(20)} (2020), 6947--6966.
\re{SF} S. Fernando, Lie algebra modules with finite dimensional weight spaces, I, {\em Trans. Amer. Math. Soc.} {\bf 322} (1990), 757--781.
\re{JEH} J. E. Humphreys, Representations of Semisimple Lie Algebras in the BGG Category $\mathcal{O}$, American Mathematical Society, 2008.
\re{BK}B. Kostant, On Whittaker vectors and representation theory, {\em Invent. Math.} {\bf48 (2)} (1978), 101--184.
\re{ML} M. Lau, Classification of Harish-Chandra modules for current algebras, {\em Proc. Amer. Math. Soc.} {\bf146} (2018), 1015--1029.
\re{LMZ} R. Lu, V. Mazorchuk, K. Zhao, On simple modules over conformal Galilei algebras, {\em J.
Pure Appl. Algebra} {\bf218(10)} (2014), 1885--1899.
\re{M} O. Mathieu, Classification of irreducible weight modules, {\em Ann. Inst. Fourier (Grenoble)} {\bf50 (2)} (2000), 537--592.
\re{VM}V. Mazorchuk, Lectures on $\mathfrak{sl}_{2}(\C)$-modules, Imperial College Press, London, 2010.
\re{MS}V. Mazorchuk, C. S\"{o}derberg, Category $\mathcal{O}$ for Takiff $\mathfrak{sl}_{2}$. {\em J. Math. Phys.} {\bf60} (2019), 111702.
\re{N} J. Nilsson, Simple $\mathfrak{sl}_{n+1}$-module structures on $\mathcal{U}(\mathfrak{h})$, {\em J. Algebra} {\bf424} (2015), 294--329.
\re{N1} J. Nilsson, $\mathcal{U}(\mathfrak{h})$-free modules and coherent families, {\em J. Pure Appl. Algebra} {\bf220(4)} (2016), 1475--1488.
\re{SYZ}Y. Su, X. Yue, X. Zhu, Simple non-weight modules over Lie superalgebras of Block type, Preprint arXiv:2101.10606.
\re{T}S. J. Takiff, Rings of invariant polynomials for a class of Lie algebras, {\em Trans. Amer. Math. Soc.} {\bf160} (1971), 249--262.
\re{TZ} H. Tan, K. Zhao, $\mathcal{W}_{n}^{+}$- and $\mathcal{W}_{n}$-module structures on $U(\mathfrak{h}_{n})$, {\em J. Algebra}  {\bf424} (2015), 357--375.
\re{TZ1} H. Tan, K. Zhao, Irreducible modules over Witt algebras $\mathcal{W}_{n}$ and over $\mathfrak{sl}_{n+1}(\C)$, {\em Algebras Represent. Theory} {\bf21(4)} (2018), 787--806.
\re{Webster}B. Webster,  Gelfand-Tsetlin modules in the Coulomb context,
Preprint arXiv:1904.05415.
\re{BJW} B. J. Wilson, Highest-weight theory for truncated current Lie algebras, {\em J. Algebra} {\bf336} (2011), 1--27.
\re{YYX} H. Yang, Y. Yao, L. Xia, On non-weight representations of the $N = 2$ superconformal algebras, {\em J. Pure Appl. Algebra} {\bf225} (2021), 106529
\re{YYX1} H. Yang, Y. Yao, L. Xia, A family of non-weight modules over the super-Virasoro algebras, {\em J. Algebra} {\bf547} (2020), 538--555.
\end{thebibliography}
 \end{document}